\documentclass[
 a4paper,
 11pt,
 reqno
]{amsart}

\usepackage[english]{babel}              
\usepackage[utf8]{inputenc}
\usepackage[T1]{fontenc}

\usepackage{lmodern}

\usepackage{amsmath}        
\usepackage{mathtools}      
\usepackage{amssymb}  
\usepackage{amsthm}
\usepackage{stackengine}
\usepackage{bbm}
\usepackage{mathrsfs}
\usepackage{nicematrix}

\numberwithin{equation}{section}

\usepackage{fancyhdr}
\usepackage{geometry}
\usepackage{graphicx}
\usepackage{tikz}
\usetikzlibrary{hobby, shapes.misc, decorations.pathmorphing, patterns}
\usepackage{pgfplots}       
\pgfplotsset{compat=1.10, ticks=none}
\usepgfplotslibrary{fillbetween}

\newcommand{\abs}[1]{\left\lvert {#1} \right\rvert}
\newcommand{\norm}[1]{\left\lVert {#1} \right\rVert}
\newcommand{\dpa}[1]{\left\langle {#1} \right\rangle}
\newcommand{\rb}[1]{\left( {#1} \right)}
\renewcommand{\sb}[1]{\left[ {#1} \right]}
\newcommand{\cb}[1]{\left\{ {#1} \right\}}

\newcommand{\R}{\mathbb R}

\newcommand{\N}{\mathbb N}
\newcommand{\Z}{\mathbb Z}
\newcommand{\T}{\mathbb T}
\newcommand{\F}{\mathcal F}

\newcommand{\St}{\mathcal S}

\renewcommand{\S}{\mathbb S}
\renewcommand{\P}{\mathscr P}
\newcommand{\cc}{\subset \subset}
\newcommand{\Ha}{\mathcal{H}}

\newcommand{\1}{\mathbbm{1}}
\newcommand{\defeq}{\vcentcolon =}

\newcommand{\loc}{\textup{loc}}

\newcommand\restrict[1]{\raisebox{-.3ex}{$\Big|$}_{#1}}
\newcommand{\mres}{\mathbin{\vrule height 0.8ex depth 0pt width 0.1ex\vrule height 0.1ex
depth 0pt width 0.8ex}}

\DeclareMathOperator{\de}{\, d \hspace{- 2pt}}
\DeclareMathOperator{\Id}{Id}

\DeclareMathOperator{\dive}{div}

\DeclareMathOperator{\reb}{\partial^* \hspace{- 2pt}}

\usepackage{url}
\usepackage[colorlinks = true, linkcolor= blue, citecolor = red, urlcolor = blue]{hyperref}
\usepackage{biblatex}
\usepackage[nameinlink, capitalise, noabbrev]{cleveref}
\usepackage{csquotes}
\theoremstyle{plain}
\newtheorem*{thm*}{Theorem}
\newtheorem{thm}{Theorem}[section]
\newtheorem{lem}[thm]{Lemma}
\newtheorem{prop}[thm]{Proposition}
\newtheorem{cor}[thm]{Corollary}
\theoremstyle{definition}
\newtheorem{defi}[thm]{Definition}
\newtheorem*{not*}{Notation}

\newtheorem{rem}[thm]{Remark}

\title[Minimality of the lamella for the anisotropic Ohta-Kawasaki energy]
{Local and global minimality of the lamella for the anisotropic Ohta-Kawasaki energy}
\author{Alberto Fiorini}

\begin{document}

\begin{abstract}

    In this paper we consider the volume-constrained minimization of a variant of the Ohta-Kawasaki functional 
    with an anisotropic surface energy replacing the standard perimeter.
    Following and suitably adapting the second variation approach devised in \cite{AFM13},
    we prove local minimality results for the horizontal lamellar configuration, in analogy with the isotropic case,
    under the assumption that the anisotropy is uniformly elliptic.
    If instead the Wulff shape of the anisotropy has upper and lower horizontal facets,
    we prove that the lamella exhibits a rigid behavior and is an isolated local minimizer for all parameter values.
    We conclude by showing some global minimality results, mostly focusing on the planar case.
    
\end{abstract}

\maketitle

\setcounter{tocdepth}{1}
\tableofcontents

\section{Introduction}

The Ohta-Kawasaki energy is a well-established functional used
to model microphase separation in diblock copolymers melts, arising as the $\Gamma$-limit of
a functional first introduced by Takao Ohta and Kyozi Kawasaki in \cite{OK86};
equilibrium configurations then correspond to volume-constrained minimizers of such functional.
Given a set $E \subset \T^n$, where $\T^n = \R^n / \Z^n$ is the $n$-dimensional flat torus of unit volume, $n \ge 2$, 
the Ohta-Kawasaki energy of $E$ is defined as 
\begin{equation*}
    J(E) \defeq P_{\T^n}(E) + \gamma \F(E) = P_{\T^n}(E) + \gamma \int_{\T^n} \rb{ \int_{\T^n} G(x, y) u_E(x) u_E(y) \de x} \de y.
\end{equation*}
Here $P_{\T^n}(E)$ denotes the perimeter of $E$ in the torus, $\gamma \ge 0$ is a fixed parameter,
$u_E \in BV(\T^n)$ is the phase function of $E$, defined as $u_E \defeq \1_E - \1_{E^c}$,
and $G$ is the Green's function on the torus for $- \Delta$ (see \cref{sec:2} for more details).
The two terms compete with each other when subject to a volume constraint: the first compels the sets to have minimal interface area 
with their complement, whereas the latter pushes for oscillation of the two phases.

In this paper we consider a generalization of the former energy by replacing the isotropic surface energy $P_{\T^n}$ 
with the anisotropic perimeter $P_{\T^n}^\phi$ corresponding to a surface tension $\phi$:
\begin{equation} \label{eq:1.anOK}
    J^\phi(E) = P_{\T^n}^\phi(E) + \gamma \F(E) = P_{\T^n}^\phi(E) + \gamma \int_{\T^n} \rb{ \int_{\T^n} G(x, y) u_E(x) u_E(y) \de x} \de y.
\end{equation}
We call $J^\phi$ the \textit{anisotropic Ohta-Kawasaki energy}.
Similarly to the isotropic case, the problem we study is the volume-constrained minimization of $J^\phi$, 
that is, we look for minimizers of
\begin{equation} \label{eq:1.anOKprob}
    \inf \Big\{ J^\phi(E) : E \subset \T^n, \ \abs{E} = M \Big\}, \qquad M \in (0, 1),
\end{equation}
where $\abs{E}$ denotes the Lebesgue measure of $E$.

\begin{figure}
    \centering
    \begin{tikzpicture} [scale = 1.7]
        \fill[gray!20] (- 1, - 1) -- (- 1, 1) -- (1, 1) -- (1, - 1) -- cycle;
        \fill[opacity = 0.5, pattern = crosshatch] (- 1, - 0.4) rectangle (1, 0.4);
        \draw[dashed] (- 1, 0.4) -- (1, 0.4);
        \draw[dashed] (- 1, - 0.4) -- (1, - 0.4);
        \draw (- 1, - 1) -- (- 1, 1) -- (1, 1) -- (1, - 1) -- cycle;
        \draw[thick, ->](- 1, - 0.00001) -- (- 1, 0.00001);
        \draw[thick, ->](1, - 0.00001) -- (1, 0.00001);
        \draw[thick, ->](- 0.00001, - 1) -- (0.00001, - 1);
        \draw[thick, ->](- 0.00001, 1) -- (0.00001, 1);
        \draw(1.3, 0.8) node{$\T^2$};
        \draw(- 1.2, 0.4) node{$\ell_2$};
        \draw(- 1.2, - 0.4) node{$\ell_1$};
        \draw(0.5, - 0.1) node{$L$};
    \end{tikzpicture}
    \caption{A lamella $L$ in the torus $\T^2$ for some $\ell_1, \ell_2 \in (0, 1)$, $\ell_1 < \ell_2$.}
    \label{fig:1.lamella}
\end{figure}
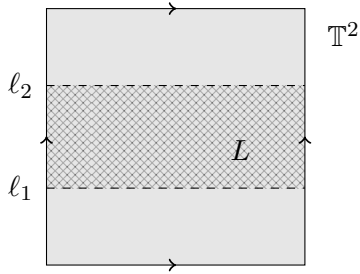

In particular, we focus on the local and global minimality properties of (single) lamellar configurations in $\T^n$, that is, 
up to translations, sets of the form $\T^{n - 1} \times (\ell_1, \ell_2)$ 
for some $\ell_1, \ell_2 \in (0, 1)$, $\ell_1 < \ell_2$, see \cref{fig:1.lamella}.
Any such set will be called \textit{lamella}, often denoted by $L$ in the following.

The minimality of lamellar configurations in the isotropic case was first investigated by Sternberg and Topaloglu in \cite{ST11}. 
They showed that, in the planar setting, a single lamella is a global minimizer of $J$ for a range of masses close to $\frac{1}{2}$ 
and for small values of the parameter $\gamma$. 
This result was later rederived and extended to three dimensions by Acerbi, Fusco, and Morini in \cite{AFM13} by means of a second variation approach. 
In particular, they proved that strict stability of a critical set for the functional $J$, that is, positivity of its second variation, 
implies local minimality with respect to small $L^1$-perturbations. 
Based on this criterion, they further established, in any dimension and for every value of $\gamma$, 
local minimality of multiple lamellar configurations (i.e., union of parallel lamellae) with a sufficiently large number of strips.
In a slightly different context, Morini and Sternberg proved in \cite{MS14} the minimality of lamellar patterns in thin domains of 
the type $(0, \varepsilon) \times (0, 1)$ with Neumann boundary conditions, without assuming that the perimeter term dominates the nonlocal energy.
Interestingly, Ren and Wei \cite{RW05} constructed stable, wriggled lamellar solutions of the Euler-Lagrange equation for the free energy $J$, 
which bifurcate from the perfect lamellar solution; wriggled lamellae had also been observed by Muratov \cite{Mur02}.
We also mention the works \cite{ACC22, DR18, GS16, GR19}, dealing with minimality of 
lamellar configurations for other functionals with local and nonlocal interaction. 
It is generally believed that minimizers of $J$ should display a periodic pattern: 
a rigorous, partial result in this direction is proved by Alberti, Choksi, and Otto in \cite{ACO09} 
(see also \cite{Mul93, RW03} for the one-dimensional case, and \cite{Cri18} for 
the construction of periodic local minimizers for $J$ close to stable, constant mean curvature surfaces in the torus).

In this paper we address the study of the minimality of lamellae in the anisotropic case. 
In the first main result (\cref{thm:2.strictminL1}) we prove, under the assumption of uniform ellipticity of $\phi$ (see \cref{defi:2.uniell}),
that the lamella is an isolated local minimizer for $J^\phi$ for small $L^1$-perturbations whenever the parameter $\gamma$ is small enough. 
The proof of this result follows the second variation approach introduced in \cite{AFM13} for the isotropic case: 
first we prove that positive second variation for the lamella implies minimality with respect to regular deformations 
(as formulated in \cref{thm:4.C1infty});
then, we refine this result as to hold for sets sufficiently close to the lamella in the $L^1$-sense, 
by an improved convergence argument based on the regularity theory for almost minimizers of the anisotropic perimeter.

We observe a more rigid behavior if we require instead that $\phi$ is horizontally flat, 
namely its Wulff shape presents two horizontal facets (see \cref{defi:2.horiflat}): 
indeed, we show in \cref{thm:2.lamminL1hor} that in this case the lamella is an isolated local minimizer for any choice of the parameter $\gamma$.
The proof exploits the previous second variation approach combined with an approximation argument as in \cite{Bon13}.
This result echoes a strong rigidity property of (quasi)-minimizers of the anistropic perimeter for crystalline surface tensions, observed by Figalli and Maggi in \cite[Theorem 7]{FM11}: they proved that in the planar case the normal vectors to such minimizers 
must align to the normal vectors of the Wulff shape. This was later extended to any dimension by Figalli and Zhang in \cite[Corollary 1.4]{FZ22}.

Concerning the global minimality, in \cref{thm:2.globmin2} we show that, in dimension $n = 2$, any uniformly elliptic surface tension 
whose Wulff shape has minimal width along the vertical direction (a condition encoded in \eqref{eq:7.minilam})
admits the horizontal lamella as global minimizer for \eqref{eq:1.anOKprob} 
when the volume constraint $M$ is close to $1/2$ (with explicit minimality range) and the parameter $\gamma > 0$ is sufficiently small.
This is again inspired by \cite{AFM13}, where the authors use well known results on minimizers for 
the isoperimetric problem with periodic boundary conditions; we have not found any reference in the literature for the periodic isoperimetric problem in the anisotropic case, for which we obtain a partial result of independent interest in \cref{thm:7.wulffminell}.

The paper is organized as follows.
In \cref{sec:2} we fix the notation and formally describe the variational problem just outlined,
stating the main results.
\cref{sec:3} is devoted to the computations of the first and second variation formulae of $J^\phi$,
which we need in \cref{sec:4} to partially prove the result of local minimality of the lamella 
in the regular case, which is then completed in \cref{sec:5}.
This is applied in \cref{sec:6} to prove the statement for a horizontally flat surface tension.
Finally, in \cref{sec:7} we study global minimality of the lamella in the planar setting.

\section{Setting and main results} \label{sec:2}

Here we fix the notation used in the remainder of this paper.
Then we formally define both $P_{\T^n}^\phi$ and $\F$ 
in order to have a clear definition of $J^\phi$ in \eqref{eq:1.anOK}.
Gathering all the necessary tools, we also give the statements of the main results.

Let $n \ge 2$. $\T^n$ denotes the $n$-dimensional flat torus of unit volume, that is, 
the quotient set of $\R^n$ obtained by identifying $x$ and $y$ whenever $y - x \in \Z^n$.
On $\T^n$ we have the usual function spaces $C^k(\T^n)$ and $W^{k, p}(\T^n)$ for 
$k \in \N \cup \cb{+ \infty}$ and $p \in [1, \infty]$,
that may be identified with the subspaces of $C^k(\R^n)$ and $W_\loc^{k, p}(\R^n)$ 
consisting of functions which are one-periodic along the coordinate directions 
(that is, functions $u : \R^n \to \R$ with $u(x + e_i) = u(x)$ for any $e_i$ in the canonical basis).
We write $x = (x', x_n)$ with $x' \in \T^{n - 1}$ to denote an element of $\T^n$, and indicate with $\nabla_{x'}$
and $\nabla_{x'}^2$ the gradient and the Hessian matrix with respect to the first $n - 1$ coordinates. 
In the following we often identify $E \subset \T^n$ with any of its translations,
since as we will later see $J^\phi$ is translation-invariant.
Therefore it makes sense to define an $L^1$-distance $d_{L^1}$ on the torus:
\begin{equation} \label{eq:2.distL1}
    d_{L^1}(E, F) \defeq \min_{x \in \T^n} \big| E \Delta (F + x) \big|, 
    \qquad \qquad E, F \subset \T^n.
\end{equation}
Here and in the following, for any sequence $\cb{E_h}_{h \in \N}$ of subsets of $\T^n$, $E \subset \T^n$,
we write $E_h \to E$ if $d_{L^1}(E_h, E) \to 0$ as $h \to \infty$.

\subsection{The anisotropic perimeter}

We start recalling some basic facts about the perimeter and its anisotropic counterpart, specifically in the case of the torus;
all the following properties are a consequence of well known results obtained in the classical Euclidean setting,
on noting that any set of finite perimeter $E \subset \T^n$ may be periodically extended 
as to obtain a set of locally finite perimeter $E_{\R^n} \subset \R^n$
(see \cite{AFP00} and \cite{Mag12} for a complete treatise).

The \textit{total variation} of a function $u \in L^1(\T^n)$ is defined as
\begin{equation*}
    \abs{D u}(\T^n) \defeq \sup \cb{\int_{\T^n} u \dive T : T \in C^1(\T^n; \R^n), \ \abs{T} \le 1},
\end{equation*}
and $u$ is said to be of \textit{bounded variation} ($u \in BV (\T^n)$) if $\abs{D u}(\T^n)$ is finite.
Then, a measurable set $E \subset \T^n$ has \textit{finite perimeter} if its characteristic function $\1_E$ belongs to $BV (\T^n)$,
defined as $P_{\T^n}(E) \defeq \abs{D \1_E} (\T^n)$.
For a set of finite perimeter $E \subset \T^n$ the \textit{Generalized Divergence Theorem} holds true:
\begin{equation*}
    \int_E \dive T = \int_{\reb E} T \cdot \nu_E \de \Ha^{n - 1} 
    \qquad \qquad \text{for any } T \in C^1(\T^n; \R^n),
\end{equation*}
where $\reb E$ is the reduced boundary of $E$,
$\nu_E : \reb E \to \S^{n - 1} = \cb{x \in \R^n : \abs{x} = 1}$ is the outer unit normal to $E$ 
and $\Ha^{n - 1}$ is the $(n - 1)$-dimensional Hausdorff measure.

To define an anisotropic surface energy we recall that 
a function $\phi : \R^n \to [0, + \infty)$ is said to be a \textit{surface tension}
if it enjoys the following properties:
\begin{itemize}
    \item 
        $\phi(\lambda x) = \lambda \phi(x)$ for any $x \in \R^n$ and $\lambda \ge 0$ 
        (\textit{positive one-homogeneity});
        
    \item 
        $\phi(x + y) \le \phi(x) + \phi(y)$ for any $x, y \in \R^n$ (\textit{subadditivity});

    \item
        $\phi(x) > 0$ for any $x \in \R^n \setminus \cb{0}$ (\textit{coercivity}).
        
\end{itemize}
Under the assumption of positive one-homogeneity, subadditivity is actually equivalent to convexity,
and in particular any surface tension is continuous. 
As such, the last property is equivalent to requiring that there exists $\alpha > 0$ 
with $\phi(x) \ge \alpha \abs{x}$ for any $x \in \R^n$.

\begin{defi}

    Let $\phi : \R^n \to [0, + \infty)$ be a surface tension.
    For any set of finite perimeter $E \subset \T^n$,
    the $\phi$\textit{-surface energy} (or \textit{anisotropic perimeter}) of $E$ in $\T^n$ is given by
    \begin{equation*}
        P_{\T^n}^\phi(E) \defeq \int_{\reb E} \phi(\nu_E(x)) \de \Ha^{n - 1} (x).    
    \end{equation*}
    
\end{defi}

The anisotropic version of the classical isoperimetric problem in $\R^n$ is known as the \textit{Wulff problem}:
minimize the anisotropic perimeter $P^\phi = P_{\R^n}^\phi$ among all sets $E \subset \R^n$ with volume
$\abs{E} = M$ for a fixed constant $M > 0$, i.e., 
\begin{equation} \label{eq:2.wulff}
    \inf \cb{P^\phi(E) : E \subset \R^n, \ \abs{E} = M},
\end{equation}
with the convention that $P^\phi(E) = + \infty$ if $E$ is not a set of finite perimeter.
The problem admits an explicit solution
\begin{equation*}
    W_\phi = \bigcap_{y \in \S^{n - 1}} \Big\{ x \in \R^n : x \cdot y < \phi(y) \Big\},
\end{equation*}
called \textit{Wulff shape of} $\phi$:
$W_\phi$ suitably rescaled as to satisfy the volume constraint
is the unique solution (up to translations) to the Wulff problem \eqref{eq:2.wulff} (see \cite{FM91}).
This can be equivalently stated through the \textit{Wulff inequality}: 
for every set $E \subset \R^n$ of finite perimeter and volume,
\begin{equation*}
    P^\phi(E) \ge n \abs{W_\phi}^{\frac{1}{n}} \abs{E}^{\frac{n - 1}{n}},
\end{equation*}
with equality if and only if $E = \lambda W_\phi + x$ for some $\lambda > 0$, $x \in \R^n$.
This estimate is improved in a quantitative fashion in \cite{FMP10}.

$W_\phi$ is open, bounded and convex, moreover $0 \in W_\phi$ for every surface tension $\phi$.
By the next proposition any such set is the Wulff shape for a suitable surface tension.

\begin{prop} \label{prop:2.surfchar}

    Let $K \subset \R^n$ be an open, bounded and convex set containing the origin.
    Then $K = W_\phi$, for the surface tension $\phi : \R^n \to [0, + \infty)$ defined as
    \begin{equation*}
        \phi(y) \defeq \sup \Big\{ x \cdot y : x \in K \Big\}, \qquad \qquad y \in \R^n. 
    \end{equation*}
    
\end{prop}

\begin{proof}

    One easily shows that $\phi$ is indeed a surface tension.
    Moreover for any $x \in K$ and $y \in \R^n$ it holds $x \cdot y < \phi(y)$ so that $x \in W_\phi$.
    Conversely, if $x \not \in K$ then there exists a hyperplane separating 
    $\cb{x}$ from $K$ (since they are convex), that is, there exists $y \in \R^n$ with
    \begin{equation*}
        z \cdot y \le x \cdot y \qquad \qquad \text{for any } z \in K. 
    \end{equation*}
    Normalizing $y$ we have that $y \in \S^{n - 1}$ and $\phi(y) \le x \cdot y$, implying $x \not \in W_\phi$.
    This concludes the proof that $K = W_\phi$.
    \qedhere
    
\end{proof}

Two classes of surface tensions are well-studied in literature, 
for their applications to physical systems: uniformly elliptic surface tensions, 
whose Wulff shape is a uniformly convex set with boundary of class $C^2$ (see \cite[Section 2.5]{Sch13}),
and crystalline surface tensions, of which the Wulff shape is a convex polyhedron.
The two different cases are in some sense complementary, 
since they are used to describe the shape of liquid drops and crystals, respectively (see for instance \cite{FM11}).
However, in our case we consider another class of surface tensions in place of the crystalline one,
because for our purposes the relevant property is that the Wulff shape has two facets parallel 
to a face of the torus (see \cref{fig:2.wulshapes}).

\begin{defi} \label{defi:2.uniell}

    A surface tension $\phi : \R^n \to [0, + \infty)$ 
    is said to be \textit{uniformly elliptic with constant $\lambda$} (or simply \textit{$\lambda$-elliptic}) 
    for some $\lambda > 0$ if $\phi$ is of class $C^2$ outside of the origin and 
    \begin{equation*}
        \nabla^2 \phi(x) \sb{v, v} = \rb{\nabla^2 \phi(x) v} \cdot v 
        \ge \dfrac{\lambda}{\abs{x}} \rb{\abs{v}^2 - \dfrac{(x \cdot v)^2}{\abs{x}^2}}
        \qquad \qquad \text{for any } x, v \in \R^n, \ x \ne 0,
    \end{equation*}
    or, equivalently,
    \begin{equation*}
        \nabla^2 \phi(x) \sb{v, v} \ge \lambda
        \qquad \qquad \text{for any } x, v \in \S^{n - 1}, \ x \perp v. 
    \end{equation*}
    
\end{defi}

\begin{defi} \label{defi:2.horiflat}

    A surface tension $\phi : \R^n \to [0, + \infty)$ is said to be \textit{horizontally flat} 
    if the subdifferential $\partial \phi$ of $\phi$ at $\pm e_n$ satisfies
    \begin{equation*}
        \Ha^{n - 1} \big( \partial \phi (\pm e_n) \big)
        = \Ha^{n - 1} \Big( \big\{ x \in \partial W_\phi : \pm x_n = \phi(\pm e_n) \big\} \Big) > 0.
    \end{equation*}
    
\end{defi}

The first equality in the definition is a direct consequence of the following result.

\begin{lem}[{\cite[Theorem 1.7.4]{Sch13}}] \label{lem:2.subsurfchar}

    The subdifferential of a surface tension $\phi : \R^n \to [0, + \infty)$ at $v \in \S^{n - 1}$ 
    is equal to the support set of $W_\phi$ at $v$, that is
    \begin{equation*}
        \partial \phi(v) = \big\{ x \in \partial W_\phi : x \cdot v = \phi(v) \big\}.
    \end{equation*}
    
\end{lem}

For a uniformly elliptic surface tension this also ensures that 
$\nabla \phi(v) \in \partial W_\phi$ with $\nu_{W_\phi}(\nabla \phi(v)) = v$ for any $v \in \S^{n - 1}$.

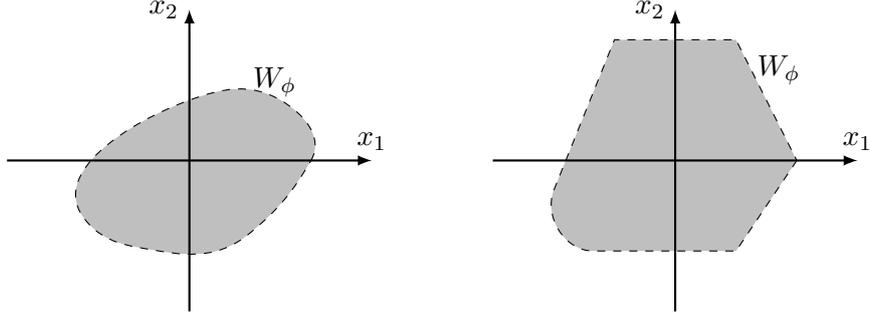
\begin{figure}
    \centering
    \begin{tikzpicture}[scale = 0.8]
        \draw[dashed, fill = gray!50] (2, 0) to[out = 60, in = - 30] (1.5, 1) to[out = 150, in = 20] (0, 1)
        to[out = 200, in = 130] (- 1.7, - 1) to[out = 310, in = 170] (- 0.5, - 1.5) to[out = - 10, in = 220] (1, - 1.2) to[out = 40, in = 240] cycle;
        \draw[thick, -latex](0, - 2.5) -- (0, 2.5) node[left]{$x_2$};
        \draw[thick, -latex](- 3, 0) -- (3, 0) node[above]{$x_1$};
        \draw(1.4, 1.3) node{$W_{\phi}$};
    \end{tikzpicture}
    \hspace{1 cm}
    \begin{tikzpicture} [scale = 0.8]
        \draw[dashed, fill = gray!50] (2, 0) -- (1, - 1.5) -- (- 1.35, - 1.5) to[out = 180, in = 250] (- 2, - 0.5) 
        -- (- 1, 2) -- (1, 2) -- cycle;
        \draw[thick, -latex](0, - 2.5) -- (0, 2.5) node[left]{$x_2$};
        \draw[thick, -latex](- 3, 0) -- (3, 0) node[above]{$x_1$};
        \draw(1.7, 1.5) node{$W_\phi$};
    \end{tikzpicture}
    \caption{The Wulff shapes for a uniformly elliptic and a horizontally flat surface tensions respectively.
    Note that the Wulff shape of a horizontally flat surface tension may or may not be regular, however it is never uniformly elliptic since 
    its Wulff shape is not uniformly convex.}
    \label{fig:2.wulshapes}
\end{figure}

\subsection{The nonlocal energy}

Let $G : (\T^n)^2 \to \R$ be the \textit{Green's function of $- \Delta$ on the torus}, i.e., for any $x \in \T^n$
the function $G_x \defeq G(x, \cdot)$ is the solution (in the sense of distributions) to
\begin{equation} \label{eq:2.green}
    \begin{dcases}
        - \Delta G_x = \delta_x - 1 \qquad \text{ in } \T^n, \\
        \int_{\T^n} G_x = 0, 
    \end{dcases}
\end{equation}
where $\delta_x$ is the Dirac delta centered at $x$.
The Green's function is symmetric and $G(x, y)$ is a function of $y - x$ for any $x, y \in \T^n$, furthermore it is integrable.

Thus, given any signed measure $\mu$ on $\T^n$ (i.e., a difference of two finite measures on $\T^n$), the function $w \in L^1(\T^n)$ defined by
\begin{equation*}
    w(x) \defeq \int_{\T^n} G_x(y) \de \mu(y), \qquad \qquad x \in \T^n,
\end{equation*}
is the unique solution (in the sense of distributions) to the problem
\begin{equation*}
    \begin{dcases}
        - \Delta w = \mu - \mu(\T^n) \qquad \text{ in } \T^n, \\
        \int_{\T^n} w = 0.
    \end{dcases}
\end{equation*}
If $\mu \in H^{- 1}(\T^n)$ and $\mu(\T^n) = 0$, Lax-Milgram's Lemma yields $w \in H^1(\T^n)$ with the equality
\begin{equation} \label{eq:2.measolest}
    \norm{\nabla w}_{L^2(\T^n)} = \norm{\mu}_{H^{- 1}(\T^n)}.
\end{equation}
In particular we have that
\begin{equation} \label{eq:2.greenpos}
    \norm{\mu}_{H^{- 1}(\T^n)}^2 = \int_{\T^n} \abs{\nabla w}^2 = - \int_{\T^n} w \Delta w = \int_{\T^n} w \de \mu
    = \int_{\T^n} \rb{\int_{\T^n} G(x, y) \de \mu(x)} \de \mu(y).
\end{equation}

Let $E \subset \T^n$ be measurable, and set $u_E \defeq \1_E - \1_{\T^n \setminus E} \in L^2(\T^n)$.
By the previous considerations $v_E : \T^n \to \R$ 
defined by $v_E(x) \defeq \int_{\T^n} G(x, y) u_E(y) \de y$, $x \in \T^n$,
belongs to $H^1(\T^n)$ and is the (weak) solution to
\begin{equation} \label{eq:2.solution}
    \begin{dcases}
        - \Delta v_E = u_E - m_E \qquad \text{ in } \T^n, \\
        \int_{\T^n} v_E = 0,
    \end{dcases}
\end{equation}
where $m_E \defeq \int_{\T^n} u_E = 2 \abs{E} - 1$.
The nonlocal term is then defined as follows.

\begin{defi}

    Given a measurable set $E \subset \T^n$, the \textit{nonlocal energy} of $E$ in $\T^n$ is defined as
    \begin{equation*}
        \F(E) \defeq \int_{\T^n} \rb{\int_{\T^n} G(x, y) u_E(x) u_E(y) \de x} \de y 
         = \int_{\T^n} v_E u_E = \int_{\T^n} \abs{\nabla v_E}^2.
    \end{equation*}
    
\end{defi}

\begin{rem} \label{rem:2.solreg}

    For any $p \in [1, \infty)$ one has $u_E \in L^p(\T^n)$, thus the Calder\'on-Zygmund inequality
    (see \cite[Theorem 9.9]{GT77}) and the periodic setting ensure that $v_E \in W^{2, p}(\T^n)$ with
    \begin{equation*}
        \norm{v_E}_{W^{2, p}(\T^n)} \le C_1 \norm{u_E - m_E}_{L^p(\T^n)}
    \end{equation*}
    for some constant $C_1 = C_1(p, n) > 0$.
    In particular, for any $p > n$, the Sobolev embedding $W^{2, p}(\T^n) \hookrightarrow C^1(\T^n)$ is compact,
    which in turn implies that $v_E \in C^1(\T^n)$ and
    \begin{equation} \label{eq:2.solestimates1}
        \norm{v_E}_{C^1(\T^n)} \le C_2 \norm{u_E - m_E}_{L^p(\T^n)} \le 2 C_2,
    \end{equation}
    for some constant $C_2 = C_2(p, n) > 0$, where the last inequality follows by noting that
    \begin{equation*}
        \norm{u_E - m_E}_{L^p(\T^n)} \le \norm{u_E}_{L^p(\T^n)} + \norm{m_E}_{L^p(\T^n)} = 1 + \abs{m_E} \le 2.
    \end{equation*}
    By the same argument, for any measurable set $F \subset \T^n$ one gets that
    \begin{equation} \label{eq:2.solestimates2}
        \begin{split}
            \norm{v_E - v_F}_{C^1(\T^n)} &\le C_2 \norm{u_E - m_E - (u_F - m_F)}_{L^p(\T^n)} \\
            &\le C_2 \rb{ \norm{u_E - u_F}_{L^p(\T^n)} + \norm{m_E - m_F}_{L^p(\T^n)}} \\
            &\le 2^p C_2 \abs{E \Delta F} + 2 C_2 \Big| \abs{E} - \abs{F} \Big|
            \le 2^{p + 1} C_2 \abs{E \Delta F}.
        \end{split}
    \end{equation}
    As a simple consequence we obtain the following well-known lemma, see for instance \cite[Lemma 2.6]{AFM13}.
    
\end{rem}

\begin{lem} \label{lem:2.nonlocont}

    There exists a constant $C_3 = C_3(n) > 0$ such that, 
    for any pair of measurable sets $E, F \subset \T^n$, it holds
    \begin{equation*}
        \abs{\F(E) - \F(F)} \le C_3 d_{L^1}(E, F).
    \end{equation*}
    In particular, $\F$ is continuous with respect to the $L^1$-distance.
    
\end{lem}

\begin{proof}

    Up to a translation we may assume that $d_{L^1}(E, F) = \abs{E \Delta F}$,
    since $\F$ is evidently translation invariant.
    The proof is then a straightforward application 
    of \eqref{eq:2.solestimates1} and \eqref{eq:2.solestimates2} for $p = n + 1$.
    Indeed it suffices to estimate the left-hand side by
    \begin{align*}
        \abs{\F(E) - \F(F)} &= \abs{\int_{\T^n} v_E u_E - \int_{\T^n} v_F u_F} 
        \le \int_{\T^n} \Big| v_E (u_E - u_F) + u_F (v_E - v_F) \Big| \\
        &\le 2 \norm{v_E}_{C_1(\T^n)} \abs{E \Delta F} + \norm{v_E - v_F}_{C^1(\T^n)} \\
        &\le 4 C_2 \abs{E \Delta F} + 2^{n + 2} C_2 \abs{E \Delta F} \le 2^{n + 3} C_2 \abs{E \Delta F},
    \end{align*}
    so that setting $C_3 = 2^{n + 3} C_2$ we may conclude.
    \qedhere
    
\end{proof}

\subsection{Main results}

We are now ready to describe precisely the problem we study.
First, given a surface tension $\phi$ and a parameter $\gamma \ge 0$ 
we rigorously define the \textit{anisotropic Ohta-Kawasaki energy} 
$J^\phi : \P(\T^n) \to [0, + \infty]$ by 
\begin{equation} \label{eq:2.anOKexp}
    J^\phi(E) \defeq P_{\T^n}^\phi(E) + \gamma \F(E) = \int_{\reb E} \phi(\nu_E) \de \Ha^{n - 1} 
    + \gamma \int_{\T^n} \rb{ \int_{\T^n} G(x, y) u_E(x) u_E(y) \de x} \de y
\end{equation}
if $E \subset \T^n$ is a set of finite perimeter and $J^\phi(E) \defeq + \infty$ otherwise.

Since both $P_{\T^n}^\phi$ and $\F$ are translation invariant, the same is true for $J^\phi$,
meaning that $J^\phi(E) = J^\phi(E + x)$ for any $x \in \T^n$, $E \subset \T^n$.
As such, up to a quotient of $\P(\T^n)$, we identify $E \subset \T^n$ with any of its translations,
and it makes sense to study the problem \eqref{eq:1.anOKprob} employing the $L^1$-distance $d_{L^1}$ defined in \eqref{eq:2.distL1}.

\begin{rem} \label{rem:2.minexst}

    By the Direct Method of Calculus of Variations one easily proves 
    that the problem \eqref{eq:1.anOKprob} admits a minimizer:
    considering a minimizing sequence $\cb{E_h}_{h \in \N}$, by positivity of the nonlocal energy it must be
    $\sup_{h \in \N} P_{\T^n}^\phi(E_h) < + \infty$,
    but then by compactness it holds $E_h \to E$ (up to a subsequence) 
    for some $E \subset \T^n$ with finite perimeter.
    \cref{lem:2.nonlocont} ensures that the nonlocal energy is continuous with respect to the $L^1$-distance, 
    and by lower semicontinuity of the anisotropic perimeter one concludes that $E$ is a minimizer. 
    
\end{rem}

Even if existence of minimizers is easy to prove, we have no information on the geometry of a minimizer.
The goal of this paper is to study the local and global minimality of the single lamellar configurations,
following the second variation approach introduced in \cite{AFM13} for the (isotropic) Ohta-Kawasaki energy.
An essential step of the strategy consists in comparing the energy of a lamella with sufficiently smooth perturbations;
to this aim, we introduce the \textit{class of strips of volume} $M \in (0, 1)$, defined as 
\begin{equation} \label{eq:2.strip}
    \St_M \defeq \cb{(g_1, g_2) \in C^1(\T^{n - 1}) \times C^1(\T^{n - 1}) : 
    0 < g_1 < g_2 < 1, \ \int_{\T^{n - 1}} (g_2 - g_1) = M}.
\end{equation}
The reason behind such terminology is that we may associate with any couple $(g_1, g_2) \in \St_M$
a set of volume $M$, called the \textit{strip of $g_1$ and $g_2$}, given by
\begin{equation} \label{eq:2.stripdef}
    \Big\{ (x', x_n) \in \T^n : g_1(x') < x_n < g_2(x') \Big\}.
\end{equation}
With some abuse of notation, we will always identify the pair $(g_1, g_2) \in \St_M$ with the previous set,
which will also be denoted $(g_1, g_2)$ (see \cref{fig:2.striplam}).
The lamella $L$ of volume $M$ is then represented by a couple $(\ell_1, \ell_2)$, 
where $0 < \ell_1 < \ell_2 < 1$ and $\ell_2 - \ell_1 = M$.
For any strip $(g_1, g_2) \in \St_M$ we may consider a one-parameter volume-preserving 
perturbation of class $C^1$: define a normed vector space $(V, \norm{\cdot}_V)$ as 
\begin{equation*}
    V \defeq \cb{\psi \in C^1(\T^{n - 1}) : \int_{\T^{n - 1}} \psi = 0}, \qquad \qquad 
    \norm{\psi}_V \defeq \norm{\nabla \psi}_{L^2(\T^{n - 1})}, \quad \psi \in V,
\end{equation*}
so that, for any two maps $\psi_1, \psi_2 \in V$, taking $\abs{t}$ sufficiently small,
we have that the strip $(g_1 + t \psi_1, g_2 + t \psi_2)$ belongs again to $\St_M$.
The norm $\norm{\cdot}_V$ is equivalent to the $H^1$-norm on $\T^{n - 1}$ 
by means of the Poincaré-Wirtinger inequality:
\begin{equation} \label{eq:2.poiwirt}
    \int_{\T^{n - 1}} \psi^2 \le (n - 1) \int_{\T^{n - 1}} \abs{\nabla \psi}^2
    \qquad \text{for any } \psi \in H^1(\T^{n - 1}) \text{ with } \int_{\T^{n - 1}} \psi = 0.
\end{equation}
We also define $\norm{\cdot}_{V \times V}$ as the natural product norm
\begin{equation*}
    \norm{(\psi_1, \psi_2)}_{V \times V} \defeq \sqrt{\norm{\psi_1}_V^2 + \norm{\psi_2}_V^2}, \qquad \qquad (\psi_1, \psi_2) \in V \times V.
\end{equation*}
Notice that the choice of the space $V$ of admissible variations automatically excludes 
all non-trivial vertical translations, along which the functional is constant.
Then, assuming some additional regularity on the surface tension $\phi$, 
one gets as necessary conditions for the minimality of $(g_1, g_2)$ 
that the first and second variation of $J^\phi$ computed at $(g_1, g_2)$ must be zero and nonnegative, 
respectively, that is,
\begin{align*}
    \dfrac{d}{d t} J^\phi(g_1 + t \psi_1, g_2 + t \psi_2) \restrict{t = 0} = 0, & &
    \dfrac{d^2}{d t^2} J^\phi(g_1 + t \psi_1, g_2 + t \psi_2) \restrict{t = 0} \ge 0.
\end{align*}
The second variation of $J^\phi$ computed at $(g_1, g_2)$
can be explicitly expressed as a quadratic form denoted by $\partial^2 J^\phi (g_1, g_2) : V \times V \to \R$, 
which maps any $(\psi_1, \psi_2) \in V \times V$ to the 
corresponding second variation
\begin{equation*}
    \partial^2 J^\phi (g_1, g_2) \sb{\psi_1, \psi_2} 
    \defeq \dfrac{d^2}{d t^2} J^\phi(g_1 + t \psi_1, g_2 + t \psi_2) \restrict{t = 0},
\end{equation*}
which we compute explicitly in \cref{thm:3.secvar}.
A more delicate issue is whether the strict positivity of the second variation
for any nontrivial perturbation suffices to conclude local minimality for a critical strip. 
This turns out to be the case for the lamella further assuming
that $\phi$ is a uniformly elliptic surface tension (see \cref{defi:2.uniell}),
and that is the content of one of the main results of this work, proved in \cref{sec:5}.
Such result mirrors \cite[Theorem 1.1]{AFM13}, obtained by the authors in the isotropic case 
for a general strictly stable critical set.

\begin{figure}
    \centering
    \begin{tikzpicture}[scale = 1.7]
        \fill[gray!20] (- 1, - 1) -- (- 1, 1) -- (1, 1) -- (1, - 1) -- cycle;
        \fill[opacity = 0.5, pattern = crosshatch] (- 1, 0.4) to[out = 60, in = 150] (- 0.5, 0.6) to[out = - 30, in = 200] (0, 0)
        to[out = 20, in = 130] (0.5, 0.3) to[out = 310, in = 240] (1, 0.4) -- (1, - 0.3) to[out = 210, in = 40] 
        (0.5, - 0.8) to[out = 220, in = - 20] (0, - 0.4) to[out = 160, in = 50] (- 0.5, - 0.5) to[out = 230, in = 30] (- 1, - 0.3) -- cycle;
        \draw[dashed] (- 1, 0.4) to[out = 60, in = 150] (- 0.5, 0.6) to[out = - 30, in = 200] (0, 0)
        to[out = 20, in = 130] (0.5, 0.3) to[out = 310, in = 240] (1, 0.4) -- (1, - 0.3) to[out = 210, in = 40] 
        (0.5, - 0.8) to[out = 220, in = - 20] (0, - 0.4) to[out = 160, in = 50] (- 0.5, - 0.5) to[out = 230, in = 30] (- 1, - 0.3) -- cycle;
        \draw (- 1, - 1) -- (- 1, 1) -- (1, 1) -- (1, - 1) -- cycle;
        \draw[thick, ->](- 1, - 0.00001) -- (- 1, 0.00001);
        \draw[thick, ->](1, - 0.00001) -- (1, 0.00001);
        \draw[thick, ->](- 0.00001, - 1) -- (0.00001, - 1);
        \draw[thick, ->](- 0.00001, 1) -- (0.00001, 1);
        \draw(1.3, 0.8) node{$\T^2$};
        \draw(- 0.4, 0.7) node{$g_2$};
        \draw(- 0. 1, - 0.5) node{$g_1$};
    \end{tikzpicture}
    \hspace{1 cm}
    \begin{tikzpicture} [scale = 1.7]
        \fill[gray!20] (- 1, - 1) -- (- 1, 1) -- (1, 1) -- (1, - 1) -- cycle;
        \fill[opacity = 0.5, pattern = crosshatch] (- 1, - 0.4) rectangle (1, 0.4);
        \draw[dashed] (- 1, 0.4) -- (1, 0.4);
        \draw[dashed] (- 1, - 0.4) -- (1, - 0.4);
        \draw (- 1, - 1) -- (- 1, 1) -- (1, 1) -- (1, - 1) -- cycle;
        \draw[thick, ->](- 1, - 0.00001) -- (- 1, 0.00001);
        \draw[thick, ->](1, - 0.00001) -- (1, 0.00001);
        \draw[thick, ->](- 0.00001, - 1) -- (0.00001, - 1);
        \draw[thick, ->](- 0.00001, 1) -- (0.00001, 1);
        \draw(1.3, 0.8) node{$\T^2$};
        \draw(- 1.2, 0.4) node{$\ell_2$};
        \draw(- 1.2, - 0.4) node{$\ell_1$};
        \draw(0.5, - 0.1) node{$L$};
    \end{tikzpicture}
    \caption{A strip of volume $M \in (0, 1)$ in $\T^2$ for some couple $(g_1, g_2) \in \St_M$. 
    A lamella of volume $M$ is then a strip for two constant functions $\ell_1$ and $\ell_2$ with $(\ell_1, \ell_2) \in \St_M$.}
    \label{fig:2.striplam}
\end{figure}
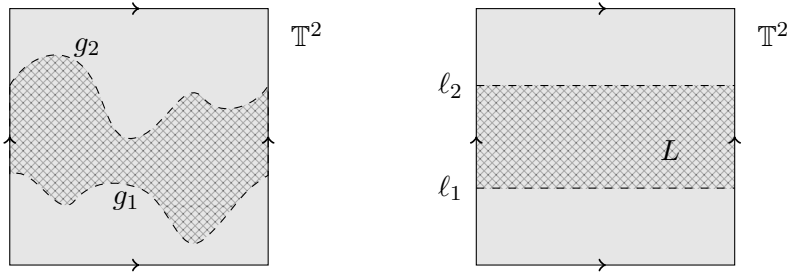

\begin{thm} \label{thm:2.strictminL1}

    Let $L$ be the lamella of volume $M \in (0, 1)$, and let $\phi : \R^n \to [0, + \infty)$ be a uniformly elliptic surface tension.
    Assume that the lamella is \emph{strictly stable} for $J^\phi$, that is,
    \begin{equation} \label{eq:2.strictstab}
        \partial^2 J^\phi(L)[\psi_1, \psi_2] > 0 \qquad \qquad 
        \text{for any } (\psi_1, \psi_2) \in V \times V \setminus \cb{(0, 0)}.
    \end{equation}
    Then there exists $\delta > 0$ with
    \begin{equation*}
        J^\phi(E) \ge J^\phi(L) + \frac{C_L^\phi}{16 (n - 1)} \big( d_{L^1}(E, L) \big)^2
    \end{equation*}
    for any $E \subset \T^n$ with $\abs{E} = M$ and $d_{L^1}(E, L) < \delta$, where
    \begin{equation*}
         C_L^\phi \defeq \inf \Big\{ \partial^2 J^\phi(L) [\psi_1, \psi_2] : 
        \psi_1, \psi_2 \in V, \ \norm{(\psi_1, \psi_2)}_{V \times V} = 1 \Big\} > 0.
    \end{equation*}
    
\end{thm}

For the statement and the proof of this result we consider the lamella, 
since we are interested in its minimality properties, however the strategy of \cite{AFM13}
could in principle be applied to any strictly stable critical set, also in the anisotropic case.
The strict stability condition \eqref{eq:2.strictstab} can be proved to hold for the lamella whenever the parameter
$\gamma > 0$ is chosen to be small enough, so that we get the corollary below.
For this kind of result we stress the dependence on $\gamma$ by writing $J_\gamma^\phi$.

\begin{cor} \label{cor:2.lamminL1}

    Let $L$ be the lamella of volume $M \in (0, 1)$, and let $\phi : \R^n \to [0, + \infty)$ be a uniformly elliptic surface tension.
    Then there exists $\gamma_0 > 0$ such that, given $\gamma \in [0, \gamma_0)$, 
    there exists $\delta > 0$ with
    \begin{equation*}
        J_\gamma^\phi(E) \ge J_\gamma^\phi(L) + \frac{C_L^\phi}{16 (n - 1)} \big( d_{L^1}(E, L) \big)^2
    \end{equation*}
    for any $E \subset \T^n$ with $\abs{E} = M$ and $d_{L^1}(E, L) < \delta$, where 
    \begin{equation*}
         C_L^\phi \defeq \inf \Big\{ \partial^2 J^\phi(L) [\psi_1, \psi_2] : 
        \psi_1, \psi_2 \in V, \ \norm{(\psi_1, \psi_2)}_{V \times V} = 1 \Big\} > 0.
    \end{equation*}
    
\end{cor}

The corollary is set out in detail in \cref{cor:5.lamminL1}. 
An approximation argument as in \cite{Bon13} allows to obtain a minimality result for the lamella 
also in the case of horizontally flat surface tensions (see \cref{defi:2.horiflat}).
In that case the lamella exhibits a very rigid behavior: it is an isolated local minimizer 
of $J_\gamma^\phi$ for any value of the parameter $\gamma \ge 0$.

\begin{thm} \label{thm:2.lamminL1hor}

    Let $L$ be the lamella of volume $M \in (0, 1)$, 
    and let $\phi : \R^n \to [0, + \infty)$ be a horizontally flat surface tension. 
    Then, $L$ is an \textit{isolated $L^1$-local minimizer} for $J_\gamma^\phi$ 
    for any $\gamma \ge 0$; more precisely, fixing $\gamma \ge 0$, there exists $\delta > 0$ such that
    \begin{equation*}
        J_\gamma^\phi(L) < J_\gamma^\phi(E)
    \end{equation*}
    for any set $E \subset \T^n$ with $\abs{E} = M$ and $0 < d_{L^1}(E, L) < \delta$.
    
\end{thm}

The proof is given in \cref{sec:6}.
Up to this point the value of the volume parameter $M$ is not relevant for our local minimality results,
however it is very natural for it to play a role when studying global minimality, as this is the case also in the isotropic setting. 
One expects that if $0 < M << 1$ (or $0 < 1 - M << 1$) the minimizer for the Wulff problem on the torus
\begin{equation} \label{eq:2.wulffper}
    \inf \cb{P_{\T^n}^\phi(E) : E \subset \R^n, \ \abs{E} = M}
\end{equation}
cannot be the lamella.
Conversely, taking $M$ close to $1/2$, we prove that in the planar case the horizontal lamella
is a minimizer for problem \eqref{eq:2.wulffper} for uniformly elliptic surface tensions
satisfying a certain minimality condition, which is a result of independent interest.
By following the same argument of \cite[Theorem 5.1]{AFM13} we show that then it is also a minimizer 
for problem \eqref{eq:1.anOKprob} for sufficiently small values of $\gamma$.

\begin{thm} \label{thm:2.globmin2}

    Let $n = 2$, and let $\phi$ be a uniformly elliptic surface tension.
    Assume also that
    \begin{equation} \label{eq:2.minilam}
        m_\phi \defeq \phi(e_2) + \phi(- e_2) < \phi(\nu) + \phi(- \nu)
    \end{equation}
    for any $\nu \in \S^1 \setminus \cb{\pm e_2}$.
    Then $m_\phi^2 < 2 \abs{W_\phi}$, and for $\frac{m_\phi^2}{4 \abs{W_\phi}} < M < 1 - \frac{m_\phi^2}{4 \abs{W_\phi}}$ 
    the horizontal lamella $L$ of volume $M$ is the unique global minimizer (up to translations)
    for problem \eqref{eq:1.anOKprob} for $\gamma > 0$ sufficiently small.
    
\end{thm}

However, in higher dimension, the same question stays open even in the isotropic case, 
with the only exception of $n = 3$, proved in \cite[Corollary 5.5]{AFM13} without an explicit minimality range.

\section{First and second variations} \label{sec:3}

In this section we compute the first and second variation formulae of $J^\phi$ at $(g_1, g_2) \in \St_M$, 
where $\St_M$ is the class of strips of volume $M$ defined in \eqref{eq:2.strip}.
In order to do that we consider local deformations as described in \cref{sec:2}.
We note that these computations were obtained more in general for the isotropic case in \cite{AFM13, CS07} (see also \cite{Mur02}).

Recalling the identification of a pair $(g_1, g_2) \in \St_M$ with the corresponding set
defined in \eqref{eq:2.stripdef}, we clearly have $\partial (g_1, g_2) = \Gamma_{g_1} \sqcup \Gamma_{g_2}$,
where $\Gamma_{g_i}$, $i = 1, 2$, is the graph of $g_i$, i.e.,
\begin{equation*} 
    \Gamma_{g_i} = \Big\{ (x', x_n) \in \T^n : x_n = g_i(x') \Big\}.
\end{equation*}
In particular the Area Formula yields
\begin{equation*}
    \int_{\partial (g_1, g_2)} \varphi \de \Ha^{n - 1}
    = \sum_{i = 1}^2 \int_{\Gamma_{g_i}} \varphi \de \Ha^{n - 1}
    = \sum_{i = 1}^2 \int_{\T^{n - 1}} \varphi \big( x', g_i(x') \big) \ \sqrt{1 + \abs{\nabla g_i(x')}^2} \de x'
\end{equation*}
for any nonnegative Borel function $\varphi : \partial (g_1, g_2) \to [0, + \infty]$.
We also have that the outer normal vector to $(g_1, g_2)$ is given by
\begin{equation*}
    \nu_{(g_1, g_2)}(x', x_n) =
    \begin{cases}
        \dfrac{\big( \nabla g_1(x'), - 1 \big)}{\sqrt{1 + \abs{\nabla g_1(x')}^2}} 
        \qquad \text{ on } \Gamma_{g_1}, \\
        \dfrac{\big( - \nabla g_2(x'), 1 \big)}{\sqrt{1 + \abs{\nabla g_2(x')}^2}} 
        \qquad \text{ on } \Gamma_{g_2}.
    \end{cases}
\end{equation*}
As a consequence, the anisotropic perimeter of a strip $(g_1, g_2) \in \St_M$ is just
\begin{equation} \label{eq:3.aniperstrip}
    P_{\T^n}^\phi(g_1, g_2) 
    = \int_{\T^{n - 1}} \Big( \phi \big( \nabla g_1(x'), - 1 \big) 
    + \phi \big( - \nabla g_2(x'), 1 \big) \Big) \de x'.
\end{equation}

We also give the following definition in order to ease the notation.

\begin{defi}

    Let $\psi_1, \psi_2 \in V$. 
    Given $\abs{t}$ small enough, so that $0 < g_1 + t \psi_1 < g_2 + t \psi_2 < 1$, denote
    \begin{equation*}
        (g_1, g_2)_t \defeq (g_1 + t \psi_1, g_2 + t \psi_2) 
        = \Big\{ (x', x_n) \in \T^n : g_1(x') + t \psi_1(x') < x_n < g_2(x') + t \psi_2(x') \Big\}.
    \end{equation*}
    The family of sets $\cb{(g_1, g_2)_t}_t \subset \St_M$ is said to be 
    a \textit{$(\psi_1, \psi_2)$-deformation} of $(g_1, g_2)$.
    
\end{defi}

This type of deformations allows us to compute the first and the second variation formulae for the anisotropic perimeter
(see \cref{fig:3.deformation}).

\begin{prop} \label{prop:3.anivar}

    Let $\psi_1, \psi_2 \in V$. 
    Consider the $(\psi_1, \psi_2)$-deformation $\cb{(g_1, g_2)_t}_t$ of $(g_1, g_2) \in \St_M$.
    For any surface tension $\phi \in C^2(\R^{n} \setminus \cb{0})$,
    \begin{align*}
        \dfrac{d}{d t} P_{\T^n}^\phi \big( (g_1, g_2)_t \big) \restrict{t = 0} 
        &= \sum_{i = 1}^2 (- 1)^{i - 1} 
        \int_{\T^{n - 1}} \nabla \psi_i(x') \cdot 
        \nabla_{x'} \phi \rb{(- 1)^i \big( - \nabla g_i(x'), 1 \big)} \de x', \\
        \dfrac{d^2}{d t^2} P_{\T^n}^\phi \big( (g_1, g_2)_t \big) \restrict{t = 0} 
        &= \sum_{i = 1}^2 
        \int_{\T^{n - 1}} \nabla_{x'}^2 \phi\rb{(- 1)^i \big( - \nabla g_i(x'), 1 \big)} 
        \sb{\nabla \psi_i(x'), \nabla \psi_i(x')} \de x'.
    \end{align*}
    
\end{prop}

\begin{proof}

    By \eqref{eq:3.aniperstrip} we know that
    \begin{equation*}
        P_{\T^n}^\phi((g_1, g_2)_t) 
        = \int_{\T^{n - 1}} \Big( \phi \big( \nabla g_1(x') + t \nabla \psi_1(x'), - 1 \big) 
        + \phi \big( - \nabla g_2(x') - t \nabla \psi_2(x'), 1 \big) \Big) \de x'.
    \end{equation*}
    The Taylor expansion of $t \mapsto \phi(\nabla g_1(x') + t \nabla \psi_1(x'), - 1)$ at $t = 0$ gives
    \begin{multline*}
        \phi \big( \nabla g_1(x') + t \nabla \psi_1(x'), - 1 \big) 
        = \phi \big( \nabla g_1(x'), - 1 \big)
        + t \nabla \psi_1(x') \cdot \nabla_{x'} \phi \big( \nabla g_1(x'), - 1 \big) \\
        + \frac{t^2}{2} \nabla_{x'}^2 \phi \big( \nabla g_1(x'), - 1 \big) 
        \sb{\nabla \psi_1(x'), \nabla \psi_1(x')} + o(t^2),
    \end{multline*}
    and similarly
    \begin{multline*}
        \phi \big( - \nabla g_2(x') - t \nabla \psi_2(x'), 1 \big) 
        = \phi \big( - \nabla g_2(x'), 1 \big) 
        - t \nabla \psi_2(x') \cdot \nabla_{x'} \phi \big( - \nabla g_2(x'), 1 \big) \\
        + \frac{t^2}{2} \nabla_{x'}^2 \phi \big( - \nabla g_2(x'), 1 \big) 
        \sb{\nabla \psi_2(x'), \nabla \psi_2(x')} + o(t^2).
    \end{multline*}
    Integrating on $\T^{n - 1}$,
    \begin{multline*}
        P_{\T^n}^\phi \big( (g_1, g_2)_t \big)
        = P_{\T^n}^\phi (g_1, g_2)
        - t \sum_{i = 1}^2 (- 1)^i 
        \int_{\T^{n - 1}} \nabla \psi_i(x') \cdot 
        \nabla_{x'} \phi \rb{(- 1)^i \big( - \nabla g_i(x'), 1 \big)} \de x' \\
        + \frac{t^2}{2} \sum_{i = 1}^2 \int_{\T^{n - 1}} \nabla_{x'}^2 \phi 
        \rb{(- 1)^i \big(- \nabla g_i(x'), 1 \big)} 
        \sb{\nabla \psi_i(x'), \nabla \psi_i(x')} \de x' + o(t^2),
    \end{multline*}
    yielding the conclusion.
    \qedhere
    
\end{proof}

\begin{figure}
    \centering
    \begin{tikzpicture}[scale = 1.7]
        \fill[gray!20] (- 1, - 1) -- (- 1, 1) -- (1, 1) -- (1, - 1) -- cycle;
        \fill[opacity = 0.5, pattern = crosshatch] (- 1, 0.4) to[out = 60, in = 150] (- 0.5, 0.6) to[out = - 30, in = 200] (0, 0)
        to[out = 20, in = 130] (0.5, 0.3) to[out = 310, in = 240] (1, 0.4) -- (1, - 0.3) to[out = 210, in = 40] 
        (0.5, - 0.8) to[out = 220, in = - 20] (0, - 0.4) to[out = 160, in = 50] (- 0.5, - 0.5) to[out = 230, in = 30] (- 1, - 0.3) -- cycle;
        \draw[dashed] (- 1, 0.4) to[out = 60, in = 150] (- 0.5, 0.6) to[out = - 30, in = 200] (0, 0)
        to[out = 20, in = 130] (0.5, 0.3) to[out = 310, in = 240] (1, 0.4) -- (1, - 0.3) to[out = 210, in = 40] 
        (0.5, - 0.8) to[out = 220, in = - 20] (0, - 0.4) to[out = 160, in = 50] (- 0.5, - 0.5) to[out = 230, in = 30] (- 1, - 0.3) -- cycle;
        \draw[dashdotted] (- 0.5, 0.6) to[out = - 30, in = 200] (0, 0.1) to[out = 20, in = 130] (0.5, 0.2) to[out = 310, in = 240] (1, 0.4);
        \draw[dashdotted] (0.5, - 0.8) to[out = 220, in = - 20] (0, - 0.3) to[out = 160, in = 50] (- 0.5, - 0.6) to[out = 230, in = 30] (- 1, - 0.3);
        \draw[dashdotted] (- 0.5, 0.6) to[out = - 30, in = 200] (0, - 0.1) to[out = 20, in = 130] (0.5, 0.4) to[out = 310, in = 240] (1, 0.4);
        \draw[dashdotted] (0.5, - 0.8) to[out = 220, in = - 20] (0, - 0.5) to[out = 160, in = 50] (- 0.5, - 0.4) to[out = 230, in = 30] (- 1, - 0.3);
        \draw (- 1, - 1) -- (- 1, 1) -- (1, 1) -- (1, - 1) -- cycle;
        \draw[thick, ->](- 1, - 0.00001) -- (- 1, 0.00001);
        \draw[thick, ->](1, - 0.00001) -- (1, 0.00001);
        \draw[thick, ->](- 0.00001, - 1) -- (0.00001, - 1);
        \draw[thick, ->](- 0.00001, 1) -- (0.00001, 1);
        \draw(1.3, 0.8) node{$\T^2$};
        \draw(- 0.4, 0.7) node{$g_2$};
        \draw(0.7, - 0.8) node{$g_1$};
    \end{tikzpicture}
    \caption{A strip of volume $M \in (0, 1)$ in $\T^2$ for some couple $(g_1, g_2) \in \St_M$ 
    and its $(\psi_1, \psi_2)$-deformation $\cb{(g_1, g_2)_t}_t$ for some $\psi_1, \psi_2 \in V$.
    By definition the volume of the deformed sets is again $M$ and one has $(g_1, g_2)_0 = (g_1, g_2)$.}
    \label{fig:3.deformation}
\end{figure}
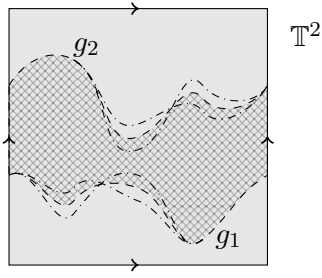
    
We now obtain the analogous formulae for the nonlocal term $\F$.
The following proposition can be found in \cite{CS07},
where it is stated for more general sets,
but we choose nonetheless to give the proof here for the reader's convenience.

\begin{prop} \label{prop:3.nonvar}

    Let $\psi_1, \psi_2 \in V$. 
    Consider the $(\psi_1, \psi_2)$-deformation $\cb{(g_1, g_2)_t}_t$ of $(g_1, g_2) \in \St_M$.
    Then,
    \begin{align*}
        \dfrac{d}{d t} \F \big( (g_1, g_2)_t \big) \restrict{t = 0} 
        &= 4 \sum_{i = 1}^2 (- 1)^i \int_{\T^{n - 1}} v_{(g_1, g_2)} \big (x', g_i(x') \big) \psi_i(x') \de x', \\
        \dfrac{d^2}{d t^2} \F \big( (g_1, g_2)_t \big) \restrict{t = 0} 
        &= \begin{multlined}[t] 8 \sum_{i, j = 1}^2 (- 1)^{i + j} 
        \int_{\T^{n - 1}} \! \rb{\int_{\T^{n - 1}} 
        \! \! G \Big( \big( x', g_i(x') \big), \big( y', g_j(y') \big) \Big) 
        \psi_i(x') \psi_j(y') \de x'} \de y' \! \! \! \! \\
        + 4 \sum_{i = 1}^2 (- 1)^i 
        \int_{\T^{n - 1}} \partial_n v_{(g_1, g_2)} \big( x', g_i(x') \big) \psi_i(x')^2 \de x',
        \end{multlined}
    \end{align*}
    where $v_{(g_1, g_2)}$ is the solution to \eqref{eq:2.solution} for the set $(g_1, g_2)$.
    
\end{prop}

\begin{proof}

    By definition of $\F$ we have that
    \begin{align*}
        \dfrac{d}{d t} \F \big( (g_1, g_2)_t \big) 
        &= \dfrac{d}{d t} \rb{\int_{\T^n} \abs{\nabla v_{(g_1, g_2)_t}}^2}
        = 2 \int_{\T^n} \nabla v_{(g_1, g_2)_t} \cdot \nabla \rb{\partial_t v_{(g_1, g_2)_t}} \\
        &= - 2 \int_{\T^n} \Delta v_{(g_1, g_2)_t} \rb{\partial_t v_{(g_1, g_2)_t}}
        = 2 \int_{\T^n} (u_{(g_1, g_2)_t} - m_{(g_1, g_2)_t}) \rb{\partial_t v_{(g_1, g_2)_t}}.
    \end{align*}
    We may write $v_{(g_1, g_2)_t}(x)$ as
    \begin{align*}
        v_{(g_1, g_2)_t}(x) &= \int_{\T^n} G(x, y) u_{(g_1, g_2)_t}(y) \de y \\
        &= \int_{(g_1, g_2)_t} G(x, y) \de y - \int_{((g_1, g_2)_t)^c} G(x, y) \de y 
        = 2 \int_{(g_1, g_2)_t} G(x, y) \de y,
    \end{align*}
    so that the derivative of $v_{(g_1, g_2)_t}$ with respect to $t$ computed at $x \in \T^n$ and evaluated at $t = 0$ takes the form
    \begin{equation} \label{eq:2.solder}
        \begin{split}
            \big( \partial_t v_{(g_1, g_2)_t}(x) \big) \restrict{t = 0}
            &= 2 \partial_t \rb{\int_{(g_1, g_2)_t} G_x(y) \de y} \restrict{t = 0} \\
            &= \partial_t \rb{\int_{\T^{n - 1}} 
            \rb{\int_{g_1(y') + t \psi_1(y')}^{g_2(y') + t \psi_2(y')} G_x(y', y_n) \de y_n} \de y'} 
            \restrict{t = 0}  \\
            &= \sum_{i = 1}^2 (- 1)^i \int_{\T^{n - 1}} G_x \big( y', g_i(y') \big) \psi_i(y') \de y'.
        \end{split}
    \end{equation}
    Whence, Fubini's Theorem yields
    \begin{align*}
        \dfrac{d}{d t} \F \big( (g_1, g_2)_t \big) \restrict{t = 0}
        &= 4 \sum_{i = 1}^2 (- 1)^i \int_{\T^n} \big( u_{(g_1, g_2)}(x) - m_{(g_1, g_2)} \big) 
        \sb{\int_{\T^{n - 1}} G_x \big( y', g_i(y') \big) \psi_i(y') \de y'} \de x \\
        &= 4 \sum_{i = 1}^2 (- 1)^i \int_{\T^{n - 1}} \!
        \rb{\int_{\T^n} \! G 
        \Big( \big( y', g_i(y') \big), x \Big) \big( u_{(g_1, g_2)}(x) - m_{(g_1, g_2)} \big) \de x} 
        \psi_i(y') \de y' \\
        &= 4 \sum_{i = 1}^2 (- 1)^i \int_{\T^{n - 1}} v_{(g_1, g_2)} \big( y', g_i(y') \big) \psi_i(y') \de y',
    \end{align*}
    concluding the proof for the first variation formula.
    As regards the second variation, observe that
    \begin{equation*}
        \dfrac{d}{d t} \F \big( (g_1, g_2)_t \big) = \dfrac{d}{d s} \F \big( (g_1, g_2)_{t + s} \big) \restrict{s = 0}
        = \dfrac{d}{d s} \F \Big( \big( (g_1, g_2)_t \big)_s \Big) \restrict{s = 0},
    \end{equation*}
    so that the computation reduces to
    \begin{equation*}
        \dfrac{d^2}{d t^2} \F \big( (g_1, g_2)_t \big) \restrict{t = 0}
        = 4 \sum_{i = 1}^2 (- 1)^i \dfrac{d}{d t} \rb{\int_{\T^{n - 1}} v_{(g_1, g_2)_t} 
        \big( x', g_i(x') + t \psi_i(x') \big) \psi_i(x') \de x'} \restrict{t = 0}.
    \end{equation*}
    Fix $i = 1, 2$. We can write
    \begin{multline*}
        \dfrac{d}{d t} \rb{\int_{\T^{n - 1}} v_{(g_1, g_2)_t} \big( x', g_i(x') + t \psi_i(x') \big) 
        \psi_i(x') \de x'} \restrict{t = 0} \\
        = \int_{\T^{n - 1}} \sb{\partial_t v_{(g_1, g_2)_t} 
        \big( x', g_i(x') \big) \restrict{t = 0} \psi_i(x') 
        + \partial_n v_{(g_1, g_2)} \big( x', g_i(x') \big) \psi_i(x')^2} \de x'.
   \end{multline*}
    By \eqref{eq:2.solder} we deduce
    \begin{equation*}
        \partial_t v_{(g_1, g_2)_t} \big( x', g_i(x') \big) \restrict{t = 0}
        = 2 \sum_{j = 1}^2 (- 1)^j \int_{\T^{n - 1}} G \Big( \big( x', g_i(x') \big), \big( y', g_j(y') \big) \Big) \psi_j(y') \de x',
    \end{equation*}
    which allows to conclude.
    \qedhere
    
\end{proof}

Putting together \cref{prop:3.anivar} and \cref{prop:3.nonvar} 
one obtains the first and second variation of the anisotropic Ohta-Kawasaki energy for strips,
which we state separately in \cref{thm:3.firstvar} and \cref{thm:3.secvar}.

\begin{thm} \label{thm:3.firstvar}

    Let $\psi_1, \psi_2 \in V$. 
    Consider the $(\psi_1, \psi_2)$-deformation $\cb{(g_1, g_2)_t}_t$ of $(g_1, g_2) \in \St_M$.
    For any surface tension $\phi \in C^2(\R^{n} \setminus \cb{0})$,
    \begin{multline*}
        \dfrac{d}{d t} J^\phi \big( (g_1, g_2)_t \big) \restrict{t = 0} \\
        = \sum_{i = 1}^2 (- 1)^i 
        \int_{\T^{n - 1}} \sb{4 \gamma v_{(g_1, g_2)} \big (x', g_i(x') \big) \psi_i(x')
        - \nabla \psi_i(x') \cdot \nabla_{x'} \phi \rb{(- 1)^i \big(- \nabla g_i(x'), 1 \big)}} \de x'. 
    \end{multline*}
    
\end{thm}

It is well known that for a strip to be a minimizer of $J^\phi$, it must be \textit{critical},
meaning that its first variation vanishes for any $(\psi_1, \psi_2)$-deformation.
A simple computation shows that any lamella $L = (\ell_1, \ell_2)$ is critical for $J^\phi$:
indeed we have
\begin{equation*}
    \dfrac{d}{d t} P_{\T^n}^\phi (L_t) \restrict{t = 0} 
    = - \sum_{i = 1}^2 (- 1)^i 
    \int_{\T^{n - 1}} \nabla \psi_i(x') \cdot \nabla_{x'} \phi \rb{(- 1)^i e_n} d x' = 0,
\end{equation*}
since $\nabla_{x'} \phi \rb{(- 1)^i e_n}$ is a constant vector for $i = 1, 2$.
As regards the nonlocal energy, by invariance of the lamella along horizontal translations,
$\T^{n - 1} \ni y' \mapsto v_L(y', \ell_i)$ is constant, therefore
\begin{equation*}
    \dfrac{d}{d t} \F(L_t) \restrict{t = 0} 
    = 4 \sum_{i = 1}^2 (- 1)^i v_L(0, \ell_i) \int_{\T^{n - 1}} \psi_i(x') \de x' = 0.
\end{equation*}

\begin{thm} \label{thm:3.secvar}

    Let $\psi_1, \psi_2 \in V$. 
    Consider the $(\psi_1, \psi_2)$-deformation $\cb{(g_1, g_2)_t}_t$ of $(g_1, g_2) \in \St_M$.
    For any surface tension $\phi \in C^2(\R^{n} \setminus \cb{0})$,
    \begin{multline} \label{eq:3.secvar}
        \partial^2 J^\phi(g_1, g_2) [\psi_1, \psi_2] 
        \defeq \dfrac{d^2}{d t^2} J^\phi((g_1, g_2)_t) \restrict{t = 0} \\
        \begin{multlined}
            = 8 \gamma \sum_{i, j = 1}^2 (- 1)^{i + j} \int_{\T^{n - 1}} 
            \rb{\int_{\T^{n - 1}} G \Big( \big( x', g_i(x') \big), \big( y', g_j(y') \big) \Big) 
            \psi_i(x') \psi_j(y') \de x'} d y'  \\
            + 4 \gamma \sum_{i = 1}^2 (- 1)^i 
            \int_{\T^{n - 1}} \partial_n v_{(g_1, g_2)} \big( x', g_i(x') \big) \psi_i(x')^2 \de x'
            \hspace{2.2 cm} \\
            + \sum_{i = 1}^2 \int_{\T^{n - 1}} \nabla_{x'}^2 \phi\rb{(- 1)^i \big( - \nabla g_i(x'), 1 \big)} 
        \sb{\nabla \psi_i(x'), \nabla \psi_i(x')} \de x'.
        \end{multlined}
    \end{multline}
    
\end{thm}

If a strip is a minimizer of $J^\phi$, besides being critical it is also \textit{stable},
namely the second variation is nonnegative for any $(\psi_1, \psi_2)$-deformation.
The lamella can be proved to be stable for $J^\phi$, assuming the surface tension to be uniformly elliptic
and the parameter $\gamma$ to be small enough, as proved in \cref{prop:3.lamstab}.
With such assumptions we can actually show it is \textit{strictly stable},
that is, the second variation is positive for any $(\psi_1, \psi_2)$-deformation different from the trivial one.

\begin{rem} \label{rem:3.greenterm}

    The Green's term in the second variation formula \eqref{eq:3.secvar} is stabilizing, 
    meaning that it is always nonnegative:
    for $i = 1, 2$ define $\underline{\psi_i} \in C^1(\T^n)$ by
    \begin{equation*}
        \underline{\psi_i} (x', x_n) \defeq \frac{\psi_i(x')}{\sqrt{1 + \abs{\nabla g_i(x')}^2}},
        \qquad \qquad (x', x_n) \in \T^n,
    \end{equation*}
    so that by setting $\mu \defeq \sum_{i = 1}^2 (- 1)^i \underline{\psi_i} \Ha^{n - 1} \mres \Gamma_{g_i}$,
    we may exploit \eqref{eq:2.greenpos} to get
    \begin{align*}
        0 &\le \int_{\T^n} \rb{\int_{\T^n} G(x, y) \de \mu(x)} \de \mu(y) \\
        &= \sum_{i, j = 1}^2 (- 1)^{i + j} 
    \int_{\T^{n - 1}} \rb{\int_{\T^{n - 1}} G \Big( \big( x', g_i(x') \big), \big( y', g_j(y') \big) \Big) 
    \psi_i(x') \psi_j(y') \de x'} \de y'.
    \end{align*}
    Moreover, $\mu \in H^{- 1}(\T^n)$: indeed, given $\Psi \in C^1(\T^n)$ such that
    $\Psi = \psi_i$ on $\Gamma_{g_i}$, $i = 1, 2$, for any $\varphi \in H^1(\T^n)$ we have
    \begin{align*}
        \abs{\dpa{\mu, \varphi}} &= \abs{\int_{\T^n} \varphi \de \mu}
        = \abs{\int_{\T^{n - 1}} \varphi(x', g_2(x')) \psi_2(x') \de x' 
        - \int_{\T^{n - 1}} \varphi(x', g_1(x')) \psi_1(x') \de x'} \\
        &= \abs{\int_{\T^{n - 1}} \rb{\int_{g_1(x')}^{g_2(x')} 
        \partial_n \Big( \varphi(x', x_n) \Psi(x', x_n) \Big) \de x_n} \de x'}
        \le \int_{\T^n} \abs{\nabla (\varphi \Psi)} \\
        &\le \norm{\Psi}_\infty \norm{\nabla \varphi}_{L^1(\T^n)} 
        + \norm{\nabla \Psi}_\infty \norm{\varphi}_{L^1(\T^n)}
        \le \norm{\Psi}_{C^1(\T^n)} \norm{\varphi}_{H^1(\T^n)}.
    \end{align*}
    
\end{rem}

\begin{rem} \label{rem:3.lamsol}

    In the case of the lamella $L = (\ell_1, \ell_2) \in \St_M$, 
    one may actually write explicitly the solution $v_L$
    of \eqref{eq:2.solution}: since the lamella is invariant under horizontal translations,
    the same is true for $v_L$, so that one can reduce to solve the one-dimensional problem
    \begin{equation*}
        \begin{dcases}
            - v'' = 2 \1_{(\ell_1, \ell_2)} - 2 M \qquad \text{ in } \S^1, \\
            \int_{\S^1} v = 0.
        \end{dcases}
    \end{equation*}
    For our ends it suffices to note 
    \begin{equation*} 
        v'(r) = 
        \begin{cases}
            2 M (r - \ell_1) + M - M^2 &\qquad \text{ on } [0, \ell_1], \\
            2 (M - 1) (r - \ell_1) + M - M^2 &\qquad \text{ on } [\ell_1, \ell_2], \\
            2 M (r - \ell_1) - M - M^2 &\qquad \text{ on } [\ell_2, 1].
        \end{cases}
    \end{equation*}
    This implies that, for any $x' \in \T^{n - 1}$,
    \begin{equation} \label{eq:2.solvalue}
        \partial_n v_L(x', \ell_1) = - \partial_n v_L(x', \ell_2) = M - M^2.
    \end{equation}
    
\end{rem}

\begin{prop} \label{prop:3.lamstab}

    Let $L$ be the lamella of volume $M \in (0, 1)$, 
    and let $\phi$ be a $\lambda$-elliptic surface tension for some $\lambda > 0$.
    Then, there exists $\gamma_0 = \gamma_0 (M, \lambda, n) > 0$ such that
    $L$ is strictly stable for $J_\gamma^\phi$ for any $\gamma \in [0, \gamma_0)$.
    
\end{prop}

\begin{proof}

    Let $\psi_1, \psi_2 \in V$, and consider 
    the $(\psi_1, \psi_2)$-deformation $\cb{L_t}_t$ of $L = (\ell_1, \ell_2)$.
    By means of \eqref{eq:2.solvalue}, the second variation for the lamella takes the following form:
    \begin{multline*}
        \partial^2 J_\gamma^\phi(L)[\psi_1, \psi_2]
        = 8 \gamma \sum_{i, j = 1}^2 (- 1)^{i + j} \int_{\T^{n - 1}} 
        \rb{\int_{\T^{n - 1}} G \Big( \big( x', \ell_i \big), \big( y', \ell_j) \big) \Big) 
        \psi_i(x') \psi_j(y') \de x'} d y' \\
        + 4 \gamma \rb{M^2 - M} \sum_{i = 1}^2 \int_{\T^{n - 1}} \psi_i(x')^2 \de x' \\
        + \sum_{i = 1}^2 \int_{\T^{n - 1}} \nabla_{x'}^2 \phi\rb{(- 1)^i e_n} 
        \sb{\nabla \psi_i(x'), \nabla \psi_i(x')} \de x'.
    \end{multline*}
    By \cref{rem:3.greenterm} the Green's term is nonnegative so that we may ignore it.
    As regards the third term, by $\lambda$-ellipticity of $\phi$ we have that, for $x' \in \T^{n - 1}$ 
    and $i = 1, 2$,
    \begin{equation*}
        \nabla_{x'}^2 \phi\rb{(- 1)^i e_n} \sb{\nabla \psi_i(x'), \nabla \psi_i(x')}
        \ge \lambda \abs{\nabla \psi_i (x')}^2,
    \end{equation*}
    and we infer
    \begin{equation*}
        \sum_{i = 1}^2 \int_{\T^{n - 1}} \nabla_{x'}^2 \phi\rb{(- 1)^i e_n} 
        \sb{\nabla \psi_i(x'), \nabla \psi_i(x')} \de x'
        \ge \lambda \sum_{i = 1}^2 \norm{\nabla \psi_i}_{L^2(\T^{n - 1})}^2.
    \end{equation*}
    Therefore, by the Poincaré-Wirtinger inequality \eqref{eq:2.poiwirt},
    \begin{align*}
        \partial^2 J_\gamma^\phi(L)[\psi_1, \psi_2]
        &\ge 4 \gamma \rb{M^2 - M} (n - 1) \sum_{i = 1}^2 \norm{\nabla \psi_i}_{L^2(\T^{n - 1})}^2
        + \lambda \sum_{i = 1}^2 \norm{\nabla \psi_i}_{L^2(\T^{n - 1})}^2 \\
        &= \rb{4 \gamma \rb{M^2 - M} (n - 1) + \lambda} \sum_{i = 1}^2 \norm{\nabla \psi_i}_{L^2(\T^{n - 1})}^2.
    \end{align*}
    Defining $\gamma_0 > 0$ as
    \begin{equation} \label{eq:3.gamma0}
        \gamma_0 \defeq \dfrac{\lambda}{4 (M - M^2) (n - 1)},
    \end{equation}
    the assertion holds true.
    \qedhere 
    
\end{proof}

\begin{rem} \label{rem:3.weakuniell}

    Uniform ellipticity is not strictly necessary in the previous result:
    Indeed, one can require directly that $\phi \in C^2(\R^n \setminus \cb{0})$
    and that there exists $\lambda > 0$ with
    \begin{equation} \label{eq:3.weakuniell}
        \nabla_{x'}^2 \phi\rb{\pm e_n} \sb{v', v'} \ge \lambda \abs{v'}^2
    \end{equation}
    for any $v' \in \R^{n - 1}$.
    In particular we may choose $\gamma_0$ defined again as in \eqref{eq:3.gamma0}.    
    
\end{rem}

\section{\texorpdfstring{$C^1$}--minimality of the lamella} \label{sec:4}

We now consider a stronger distance than $d_{L^1}$, defined on $\St_M$:
the $C^1$-distance $d_{C^1} : \St_M \times \St_M \to \R$, given by 
\begin{equation} \label{eq:4.C1dist}
    d_{C^1} \big( (g_1, g_2), (h_1, h_2) \big) 
    \defeq \min_{\tau' \in \T^{n - 1}} 
    \rb{\sum_{i = 1}^2 \norm{\nabla h_i(\cdot + \tau') - \nabla g_i(\cdot)}_\infty}
\end{equation}
for any $(g_1, g_2), (h_1, h_2) \in \St_M$.
If one of the two strips is a lamella, this reduces to
\begin{equation*}
    d_{C^1} \big( (g_1, g_2), L \big) = \norm{\nabla g_1}_\infty + \norm{\nabla g_2}_\infty.
\end{equation*}
    
In this section we aim to prove a local minimality result with respect to the $C^1$-distance for a critical and strictly stable strip,
also obtaining a quantitative estimate, as stated in \cref{thm:4.C1infty}.

\begin{thm} \label{thm:4.C1infty}

    Let $(g_1, g_2) \in \St_M$ be a critical and strictly stable strip for $J^\phi$,
    and suppose that $\phi \in C^2(\R^n \setminus \cb{0})$ is a uniformly elliptic surface tension.
    Then there exists $\delta > 0$ such that
    \begin{equation}
        J^\phi(h_1, h_2) 
        \ge J^\phi(g_1, g_2) + \dfrac{C_{(g_1, g_2)}^\phi}{4} \min_{\nu' \in \T^{n - 1}}
        \rb{\sum_{i = 1}^2 \norm{\nabla h_i(\cdot + \nu') - \nabla g_i(\cdot)}_{L^2(\T^{n - 1})}^2}
    \end{equation}
    for any strip $(h_1, h_2) \in \St_M$ with $d_{C^1}((g_1, g_2), (h_1, h_2)) < \delta$, where
    \begin{equation} \label{eq:4.constant}
        C_{(g_1, g_2)}^\phi \defeq \inf \Big\{ \partial^2 J^\phi(g_1, g_2) [\psi_1, \psi_2] : 
        \psi_1, \psi_2 \in V, \ \norm{(\psi_1, \psi_2)}_{V \times V} = 1 \Big\}.
    \end{equation}
    
\end{thm}

\begin{rem} \label{rem:4.C1est}

    For any two strips $(g_1, g_2), (h_1, h_2) \in \St_M$, we may assume, up to a translation, that
    \begin{equation*}
        \sum_{i = 1}^2 \norm{h_i - g_i}_\infty 
        \le \sum_{i = 1}^2 \norm{\nabla h_i - \nabla g_i}_\infty
        = d_{C^1} \big( (g_1, g_2), (h_1, h_2) \big).
    \end{equation*}
    Indeed, if $\tau' \in \T^{n - 1}$ attains the minimum in \eqref{eq:4.C1dist}, 
    define $\tilde h_i \in C^1(\T^{n - 1})$ by
    \begin{equation*}
        \tilde h_i(x') \defeq h_i (x' + \tau') - \int_{\T^{n - 1}} (h_i - g_i), 
        \qquad \qquad x' \in \T^{n - 1}, \qquad i = 1, 2.
    \end{equation*}
    Then, up to a translation $(\tilde h_1, \tilde h_2) = (h_1, h_2)$, moreover
    \begin{equation*}
        \tilde h_i(x') - g_i(x') 
        = h_i(x' + \tau') - g_i(x') - \int_{\T^{n - 1}} (h_i - g_i)
        \le \norm{\nabla \tilde h_i - \nabla g_i}_\infty
    \end{equation*}
    for any $x' \in \T^{n - 1}$, validating the claim.

\end{rem}

For the following, we note that $(V, \norm{\cdot}_V)$ is not a Banach space,
however its completion $(H, \norm{\cdot}_H)$ is a reflexive space, where
\begin{equation*}
    H \defeq \overline V^{\norm{\cdot}_V} = \cb{ \psi \in H^1(\T^{n - 1}) : \int_{\T^{n - 1}} \psi = 0}, 
    \qquad \qquad \norm{\psi}_H = \norm{\nabla \psi}_{L^2(\T^{n - 1})}, \quad \psi \in H.
\end{equation*}
$(H \times H, \norm{\cdot}_{H \times H})$ is reflexive as well, taking the natural product norm
\begin{equation*}
    \norm{(\psi_1, \psi_2)}_{H \times H} \defeq \sqrt{\norm{\psi_1}_H^2 + \norm{\psi_2}_H^2},
    \qquad \qquad (\psi_1, \psi_2) \in H \times H.
\end{equation*}
Therefore, given $(g_1, g_2) \in \St_M$, one may easily extend $\partial^2 J^\phi(g_1, g_2)$ computed in \eqref{eq:3.secvar} to $H \times H$ by continuity.
Moreover $\partial^2 J^\phi(g_1, g_2)$ is also weakly lower semicontinuous, as stated in the following lemma.

\begin{lem}

    The second variation functional 
    $\partial^2 J^\phi(g_1, g_2) : H \times H \to \R$ is weakly lower semicontinuous,
    that is, for any sequence $\cb{(\psi_{1, h}, \psi_{2, h})}_{h \in \N} \subset H \times H$
    weakly converging to $(\psi_1, \psi_2) \in H \times H$ it holds
    \begin{equation*}
        \partial^2 J^\phi(g_1, g_2) \sb{\psi_1, \psi_2} 
        \le \liminf_{h \to \infty} \partial^2 J^\phi(g_1, g_2) \sb{\psi_{1, h}, \psi_{2, h}}.
    \end{equation*}
    
\end{lem}

\begin{proof}

    Given a sequence $\cb{(\psi_{1, h}, \psi_{2, h})}_{h \in \N} \subset H \times H$
    weakly converging to $(\psi_1, \psi_2) \in H \times H$, by the Rellich-Kondrachov Theorem it follows that
    \begin{align*}
        \psi_{i, h} &\to \psi_i& &\text{strongly in } L^2(\T^{n - 1}), &i = 1, 2, \\
        \nabla \psi_{i, h} &\rightharpoonup \nabla \psi_i 
        & &\text{weakly in } L^2(\T^{n - 1}; \R^{n - 1}), &i = 1, 2.
    \end{align*}
    Therefore the first two terms of \eqref{eq:3.secvar} are actually weakly continuous,
    since $L^2$-convergence of $\psi_{i, h}$ to $\psi_i$, $i = 1, 2$, suffices.
    To conclude, the third term of \eqref{eq:3.secvar} is weakly lower semicontinuous: since the map 
    $L^2(\T^{n - 1}; \R^{n - 1}) \ni \tau 
    \mapsto \int_{\T^{n - 1}} \nabla_{x'}^2 \phi\rb{(- 1)^i (- \nabla g_i(x'), 1)} \sb{\tau(x'), \tau(x')} \de x'$
    is strongly continuous and convex for $i = 1, 2$, it must also be also weakly lower semicontinuous, 
    and the conclusion follows.
    \qedhere
    
\end{proof}

The next proposition shows that if a strip $(g_1, g_2) \in \St_M$ is critical and strictly stable for $J^\phi$,
then the constant defined in \eqref{eq:4.constant} is strictly positive 
and one has some kind of coercivity for the functional $\partial^2 J^\phi(g_1, g_2)$.

\begin{prop} \label{prop:4.coerc}

    Let $(g_1, g_2) \in \St_M$ be a critical and strictly stable strip for $J^\phi$,
    and suppose that $\phi \in C^2(\R^n \setminus \cb{0})$ is 
    a uniformly elliptic surface tension.
    Then, it holds
    \begin{equation} \label{eq:4.coercinf} 
        C_{(g_1, g_2)}^\phi \defeq \inf \Big\{ \partial^2 J^\phi(g_1, g_2) [\psi_1, \psi_2] : 
        \psi_1, \psi_2 \in V, \ \norm{(\psi_1, \psi_2)}_{V \times V} = 1 \Big\} > 0,
    \end{equation}
    and 
    \begin{equation*}
        \partial^2 J^\phi(g_1, g_2) [\psi_1, \psi_2] 
        \ge C_{(g_1, g_2)}^\phi \norm{(\psi_1, \psi_2)}_{V \times V}^2 
        \qquad \qquad \text{for any } (\psi_1, \psi_2) \in V \times V.
    \end{equation*}
    
\end{prop}

\begin{proof}

    Consider a minimizing sequence $\cb{(\psi_{1, h}, \psi_{2, h})}_{h \in \N} \subset V \times V$ 
    for \eqref{eq:4.coercinf}.
    By reflexivity of $H \times H$, up to a subsequence there exists $(\psi_1, \psi_2) \in H \times H$ with
    $(\psi_{1, h}, \psi_{2, h}) \rightharpoonup (\psi_1, \psi_2)$ weakly in $H \times H$ as $h \to \infty$.
    If $(\psi_1, \psi_2) \ne (0, 0)$, by weak lower semicontinuity of $\partial^2 J^\phi(g_1, g_2)$ 
    and strict stability of $(g_1, g_2)$ one obtains
    \begin{equation*}
        C_{(g_1, g_2)}^\phi = \lim_{h \to \infty} \partial^2 J^\phi(g_1, g_2) \sb{\psi_{1, h}, \psi_{2, h}}
        \ge \partial^2 J^\phi(g_1, g_2) \sb{\psi_1, \psi_2} > 0.
    \end{equation*}
    If instead $(\psi_1, \psi_2) = (0, 0)$, then again by the Rellich-Kondrachov Theorem
    $\psi_{i, h} \to 0$ strongly in $L^2(\T^{n - 1})$ as $h \to \infty$ for $i = 1, 2$. Since $\phi$ is uniformly elliptic for some constant $\lambda > 0$, it follows that
    \begin{align*}
        C_{(g_1, g_2)}^\phi &= \lim_{h \to \infty} \partial^2 J^\phi(g_1, g_2) \sb{\psi_{1, h}, \psi_{2, h}} \\
        &= \lim_{h \to \infty} \sum_{i = 1}^2 \int_{\T^{n - 1}} 
        \nabla_{x'}^2 \phi \rb{(- 1)^i \big( - \nabla g_i(x'), 1 \big)}
        \sb{\nabla \psi_{i, h}(x'), \nabla \psi_{i, h}(x')} \de x' \\
        &\ge \limsup_{h \to \infty} 
        \sum_{i = 1}^2 \int_{\T^{n - 1}} \dfrac{\lambda}{\sqrt{1 + \abs{\nabla g_i}^2}} 
        \abs{\nabla \psi_{i, h}(x')}^2 \de x' \\
        &\ge \dfrac{\lambda}{1 + \norm{\nabla g_1}_\infty + \norm{\nabla g_2}_\infty} 
        \limsup_{h \to \infty} \sum_{i = 1}^2 \norm{\nabla \psi_{i, h}}_{L^2(\T^{n - 1})}^2
        = \dfrac{\lambda}{1 + \norm{\nabla g_1}_\infty + \norm{\nabla g_2}_\infty} > 0.
    \end{align*}
    The second part of the statement follows easily 
    by noting that $\partial^2 J^\phi(g_1, g_2)$ is a quadratic form, namely,
    $\partial^2 J^\phi(g_1, g_2) [\alpha \psi_1, \alpha \psi_2] 
    = \alpha^2 \partial^2 J^\phi(g_1, g_2) [\psi_1, \psi_2]$
    for any $\alpha \in \R$, $(\psi_1, \psi_2) \in V \times V$.
    \qedhere
    
\end{proof}

To prove \cref{thm:4.C1infty} we also need the following lemma, which shows continuity of $\partial^2 J^\phi$
with respect to the $C^1$-distance $d_{C^1}$.

\begin{lem} \label{lem:4.secvarcont}

    Let $(g_1, g_2) \in \St_M$.
    Then, for any $\varepsilon > 0$ there exists $\delta = \delta(\varepsilon) > 0$ such that
    \begin{equation*}
        \Big| \partial^2 J^\phi(g_1, g_2)[\psi_1, \psi_2] - \partial^2 J^\phi(h_1, h_2)[\psi_1, \psi_2] \Big| 
        < \varepsilon \norm{(\psi_1, \psi_2)}_{V \times V}^2 
        \qquad \qquad \text{for any } (\psi_1, \psi_2) \in V \times V,
    \end{equation*}
    for each $(h_1, h_2) \in \St_M$ with $d_{C^1} ((g_1, g_2), (h_1, h_2)) < \delta$.
    
\end{lem}

\begin{proof}

    Consider a sequence $\cb{(g_{1, h}, g_{2, h})}_{h \in \N} \subset \St_M$
    with $d_{C^1} ((g_1, g_2), (g_{1, h}, g_{2, h})) \to 0$ as $h \to \infty$.
    By means of \cref{rem:4.C1est}, we may assume that
    \begin{equation*}
        \lim_{h \to \infty} \norm{g_{i, h} - g_i}_\infty = 0, \qquad
        \lim_{h \to \infty} \norm{\nabla g_{i, h} - \nabla g_i}_\infty = 0, \qquad \qquad i = 1, 2.
    \end{equation*}
    We prove that
    \begin{equation*} 
        \lim_{h \to \infty} \Big| \partial^2 J^\phi(g_1, g_2)[\psi_{1, h}, \psi_{2, h}]
        - \partial^2 J^\phi(g_{1, h}, g_{2, h})[\psi_{1, h}, \psi_{2, h}] \Big| = 0
    \end{equation*}
    for an arbitrary sequence $\cb{\psi_{1, h}, \psi_{2, h}}_{h \in \N} \in V \times V$ with 
    $\norm{(\psi_{1, h}, \psi_{2, h})}_{V \times V} = 1$ for all $h \in \N$.
    To do that, we break down $\partial^2 J^\phi(g_{1, h}, g_{2, h})$ in $3$ pieces and 
    show convergence to the corresponding terms of $\partial^2 J^\phi(g_1, g_2)$. 
    \begin{itemize}
        \item 
            Let us show that for $i = 1, 2$
            \begin{multline*}
                I_i \defeq \lim_{h \to \infty}
                \Bigg|\int_{\T^{n - 1}} \nabla_{x'}^2 \phi\rb{(- 1)^i \big( - \nabla g_i(x'), 1 \big)} 
                \sb{\nabla \psi_{i, h}(x'), \nabla \psi_{i, h} (x')} \de x' \\
                - \int_{\T^{n - 1}} \nabla_{x'}^2 \phi\rb{(- 1)^i \big( - \nabla g_{i, h}(x'), 1 \big)} 
                \sb{\nabla \psi_{i, h}(x'), \nabla \psi_{i, h} (x')} \de x' \Bigg|
                = 0.
            \end{multline*}
            Indeed it holds
            \begin{equation*}
                I_i \le \liminf_{h \to \infty}
                \norm{\nabla_{x'}^2 \phi \rb{(- 1)^i \big( - \nabla g_i(x'), 1 \big)} 
                - \nabla_{x'}^2 \phi \rb{(- 1)^i \big( - \nabla g_{i, h}(x'), 1 \big)}}_\infty
                = 0,
            \end{equation*}
            on noting that $\phi \in C^2(\R^n \setminus \cb{0})$ and 
            $\lim_{h \to \infty} \norm{\nabla g_{i, h} - \nabla g_i}_\infty = 0$.

        \item 
            We now prove that for $i = 1, 2$
            \begin{multline*}
                II_i \defeq \lim_{h \to \infty} \Bigg| \int_{\T^{n - 1}} \partial_n v_{(g_1, g_2)} 
                \big( x', g_i(x') \big) \psi_{i, h}(x')^2 \de x' \\
                - \int_{\T^{n - 1}} \partial_n v_{(g_{1, h}, g_{2, h})} 
                \big( x', g_{i, h}(x') \big) \psi_{i, h}(x')^2 \de x' \Bigg| 
                = 0.
            \end{multline*}
            By \cref{rem:2.solreg} we know that 
            $v_{(g_{1, h}, g_{2, h})} \to v_{(g_1, g_2)}$ in $C^1(\T^n)$ as $h \to \infty$,
            implying
            \begin{align*}
                II_i &\le 
                \begin{multlined}[t]
                    \liminf_{h \to \infty} 
                    \Bigg( \abs{\int_{\T^{n - 1}} \rb{\partial_n v_{(g_1, g_2)} \big( x', g_i(x') \big) 
                    - \partial_n v_{(g_1, g_2)} \big( x', g_{i, h}(x') \big)} \psi_{i, h}(x')^2 \de x'} \\
                    + \abs{\int_{\T^{n - 1}} \rb{\partial_n v_{(g_1, g_2)} \big( x', g_{i, h}(x') \big) 
                    - \partial_n v_{(g_{1, h}, g_{2, h})} \big( x', g_{i, h}(x') \big)} 
                    \psi_{i, h}(x')^2 \de x'} \Bigg)
                \end{multlined} \\
                &\le (n - 1) \liminf_{h \to \infty} \norm{v_{(g_1, g_2)} - v_{(g_{1, h}, g_{2, h})}}_{C^1(\T^n)} = 0,
            \end{align*}
            where we also used the continuity of $\partial_n v_{(g_1, g_2)}$ 
            and $\lim_{h \to \infty} \norm{g_{i, h} - g_i}_\infty = 0$,
            as well as the Poincaré-Wirtinger inequality \eqref{eq:2.poiwirt}.

        \item
            Using the same notation of \cref{rem:3.greenterm},
            set $\mu_h = \sum_{i = 1}^2 (- 1)^i \underline{\psi_{i, h}} \Ha^{n - 1} \mres \Gamma_{g_i}$
            and $\mu_h' = \sum_{i = 1}^2 (- 1)^i \underline{\psi_{i, h}} \Ha^{n - 1} \mres \Gamma_{g_ {i, h}}$.
            Noting that by the same remark both $\mu_h$ and $\mu_h'$ belong to $H^{- 1}(\T^n)$,
            we show that
            \begin{equation*}
                III \defeq \lim_{h \to \infty} \norm{\mu_h - \mu_h'}_{H^{- 1}(\T^n)} = 0.
            \end{equation*}
            By means of \eqref{eq:2.greenpos} and \eqref{eq:2.measolest} this is sufficient to prove 
            convergence for the Green's term.
            Hence, given $\varphi \in H^1(\T^n)$, it follows that
            \begin{align*}
                \abs{\dpa{\mu_h - \mu_h', \varphi}} 
                &= \abs{\int_{\T^n} \varphi \de \mu_h - \int_{\T^n} \varphi \de \mu_h'} \\
                &\le \sum_{i = 1}^2 \abs{\int_{\T^{n - 1}} 
                \Big( \varphi \big( x', g_i(x') \big) 
                - \varphi \big( x', g_{i, h}(x') \big) \Big) \psi_{i, h}(x') \de x'} \\
                &\le \sum_{i = 1}^2 \int_{\T^{n - 1}} \abs{\psi_{i, h}(x')} 
                \abs{\int_{g_{i, h}(x')}^{g_i(x')}
                \partial_n \varphi(x', x_n) \de x_n} \de x' \\
                &\le \sum_{i = 1}^2 \norm{\psi_{i, h}}_{L^2(\T^{n - 1})}
                \norm{g_{i_h} - g_i}_\infty \norm{\nabla \varphi}_{L^2(\T^n)} \\
                &\le \sqrt{n - 1} \norm{\varphi}_{H^1(\T^n)} \sum_{i = 1}^2 \norm{g_{i_h} - g_i}_\infty
            \end{align*}
            again by the Poincaré-Wirtinger inequality \eqref{eq:2.poiwirt}.
            Since $\lim_{h \to \infty} \norm{g_{i, h} - g_i}_\infty = 0$ the claim follows.
        
    \end{itemize}
    Observing that
    \begin{equation*} 
        \lim_{h \to \infty} \Big| \partial^2 J^\phi(g_1, g_2)[\psi_{1, h}, \psi_{2, h}]
        - \partial^2 J^\phi(g_{1, h}, g_{2, h})[\psi_{1, h}, \psi_{2, h}] \Big| \le I_1 + I_2 + II_1 + II_2 + III = 0,
    \end{equation*}
    we conclude the proof of the lemma.
    \qedhere
    
\end{proof}

By virtue of \cref{lem:4.secvarcont}, we can finally pass to the proof of \cref{thm:4.C1infty}.

\begin{proof}[Proof of \cref{thm:4.C1infty}]

    First of all, by \cref{prop:4.coerc} there exists $C_{(g_1, g_2)}^\phi > 0$ such that
    \begin{equation*}
        \partial^2 J^\phi(g_1, g_2) [\psi_1, \psi_2] 
        \ge C_{(g_1, g_2)}^\phi \norm{(\psi_1, \psi_2)}_{V \times V}^2
    \end{equation*}
    for any $(\psi_1, \psi_2) \in V \times V$.
    Choosing for instance $\varepsilon = C_{(g_1, g_2)}^\phi/2$, let $\delta > 0$ be the same of \cref{lem:4.secvarcont}.
    Consider $(h_1, h_2) \in \mathcal{S}_M$ with $d_{C^1}((g_1, g_2), (h_1, h_2)) < \delta$, and let $\tau' \in \T^{n - 1}$ attain the minimum in \eqref{eq:4.C1dist}.
    Define $\psi_1, \psi_2 \in V$ by
    \begin{equation*}
        \psi_i(x') \defeq h_i(x' + \tau') - g_i(x') - \int_{\T^{n - 1}} (h_i - g_i), \qquad \qquad x' \in \T^{n - 1},
        \qquad i = 1, 2.
    \end{equation*}
    The $(\psi_1, \psi_2)$-deformation $\cb{(g_1, g_2)_t}_t$ of $(g_1, g_2)$ satisfies,
    up to a translation, $(g_1, g_2)_1 = (h_1, h_2)$, and obviously $(g_1, g_2)_0 = (g_1, g_2)$.
    By considering the Taylor expansion of $J^\phi((g_1, g_2)_t)$ with integral remainder with respect to $t$, 
    one gets
    \begin{equation*}
        \begin{split}
            J^\phi(h_1, h_2) - J^\phi(g_1, g_2) &= J^\phi((g_1, g_2)_1) - J^\phi((g_1, g_2)_0) \\
            &= t \dfrac{d}{d t} J^\phi((g_1, g_2)_t) \restrict{t = 0} +
            \int_0^1 (1 - t) \dfrac{d^2}{d t^2} J^\phi \big( (g_1, g_2)_t \big) \de t \\
            &= \int_0^1 (1 - t) \partial^2 J^\phi \big( (g_1, g_2)_t \big)[\psi_1, \psi_2] \de t,
        \end{split}
    \end{equation*}
    on noting that $(g_1, g_2)$ is critical for $J^\phi$ by assumption.
    Pointing out that
    $d_{C^1} ((g_1, g_2), (g_1, g_2)_t) \le t d_{C^1} ((g_1, g_2), (h_1, h_2)) < \delta$ for any $t \in [0, 1]$,
    we may apply \cref{prop:4.coerc} and \cref{lem:4.secvarcont} to obtain
    \begin{equation*}
        J^\phi(h_1, h_2) - J^\phi(g_1, g_2) 
        \ge \dfrac{C_{(g_1, g_2)}^\phi}{2} \int_0^1 (1 - t) \norm{(\psi_1, \psi_2)}_{V \times V}^2 \de t
        = \dfrac{C_{(g_1, g_2)}^\phi}{4} \norm{(\psi_1, \psi_2)}_{V \times V}^2.
    \end{equation*}
    The statement follows, since
    \begin{align*}
        \norm{(\psi_1, \psi_2)}_{V \times V}^2 
        &= \sum_{i = 1}^2 \norm{\nabla h_i(\cdot + \tau') - \nabla g_i(\cdot)}_{L^2(\T^{n - 1})}^2 \\
        &\ge \min_{\nu' \in \T^{n - 1}} \sum_{i = 1}^2 \norm{\nabla h_i(\cdot + \nu') - \nabla g_i(\cdot)}_{L^2(\T^{n - 1})}^2.
        \qedhere
    \end{align*}
    
\end{proof}

\section{From \texorpdfstring{$C^1$}--minimality to \texorpdfstring{$L^1$}--minimality} \label{sec:5}

Combining \cref{prop:3.lamstab} and \cref{thm:4.C1infty} 
one may already obtain a $C^1$-local minimality result for the lamella 
for $\gamma$ sufficiently small.
However, in this section we refine \cref{thm:4.C1infty} by showing that the result holds 
also for the $L^1$-distance in the class of the subsets of the torus (\cref{thm:2.strictminL1}),
which may instead be used to prove $L^1$-local minimality for the lamella,
again for $\gamma$ sufficiently small (\cref{cor:2.lamminL1}).

To do that, we first recall the definition of almost minimizer for the anisotropic perimeter in $\R^n$.

\begin{defi}
    
    Given a surface tension $\phi$ and two constants $\Lambda \ge 0$ and $r_0 > 0$,
    a set of locally finite perimeter $E \subset \R^n$ is said to be a \textit{$(\phi, \Lambda, r_0)$-perimeter minimizer} if
    for any $x \in \R^n$ and $r < r_0$ it holds
    \begin{equation*}
        P^\phi \big( E; B(x, r) \big) \le P^\phi \big( F; B(x, r) \big) + \Lambda \abs{E \Delta F}
    \end{equation*}
    for any set of locally finite perimeter $F \subset \R^n$ with $E \Delta F \cc B(x, r)$.
    
\end{defi}
    
By Theorem 4.6 of \cite{Neu16} we have the following regularity result for anisotropic almost minimizers.

\begin{thm} \label{thm:5.almostreg}

    Let $\phi$ be a uniformly elliptic surface tension, 
    $\cb{E_h}_{h \in \N}$ a sequence of $(\phi, \Lambda, r_0)$-minimizers for some $\Lambda \ge 0$, $r_0 > 0$,
    and $E \subset \R^n$ a set with boundary of class $C^2$, 
    such that $\1_{E_h} \to \1_E$ in $L^1(\R^n)$ as $h \to \infty$.
    Then, for $h$ sufficiently large, there exist functions $\psi_h \in C^1(\partial E)$ such that
    \begin{equation*}
        \partial E_h = \Big\{ x + \psi_h(x) \nu_E(x) \in \R^n : x \in \partial E \Big\}
    \end{equation*}
    and $\norm{\psi_h}_{C^1(\partial E)} \to 0$ as $h \to \infty$.
    
\end{thm}

We also need the following technical lemma.

\begin{lem} \label{lem:4.boundper}

    Let $L = (\ell_1, \ell_2)$ be the lamella of volume $M \in (0, 1)$, and let $\phi$ be a surface tension.
    Then, defining a constant $C_4 = C_4(\phi, M) > 0$ by
    \begin{equation*}
        C_4 \defeq \frac{2 \max_{\S^{n - 1}} \phi}{M (1 - M)},
    \end{equation*}
    it holds
    \begin{enumerate}
        \item 
            $P_{\T^n}^\phi (L) - P_{\T^n}^\phi(F) \le C_4 \abs{F \Delta L}$;

        \item 
            $P_{\T^n}^\phi (L \cap F) - P_{\T^n}^\phi(F) \le C_4 \abs{F \setminus L}$;

        \item 
            $P_{\T^n}^\phi (L \cup F) - P_{\T^n}^\phi(F) \le C_4 \abs{L \setminus F}$;
            
    \end{enumerate}
    for any set of finite perimeter $F \subset \T^n$.
    
\end{lem}

\begin{proof}

    From \cref{prop:2.surfchar} we infer that there exist $x^+, x^- \in \partial W_\phi$ with $x_n^\pm = \pm \phi(\pm e_n)$,
    moreover, since $x^\pm \cdot y \le \phi(y)$ for any $y \in \S^{n - 1}$, 
    we also have $\abs{x^\pm} \le \phi \rb{x^\pm / \abs{x^\pm}} \le \max_{\S^{n - 1}} \phi$.
    Let $ \varphi \in C^1(\T^1)$ be a function with Lipschitz constant smaller than $\frac{1}{M (1 - M)}$
    satisfying $\varphi(\ell_1) = \min_{\S^1} \varphi = 0$ and $\varphi(\ell_2) = \max_{\S^1} \varphi = 1$.
    Defining $T \in C^1(\T^n; \R^n)$ by
    \begin{equation*}
        T(x', x_n) \defeq x^- + \varphi(x_n) (x^+ - x^-), \qquad \qquad (x', x_n) \in \T^n,
    \end{equation*}
    we see that $T(x', \ell_1) = x^-$ and $T(x', \ell_2) = x^+$ for any $x' \in \T^{n - 1}$, 
    which by convexity of $\overline W_\phi$ allows to write $T \in C^1(\T^n; \overline W_\phi)$.
    Furthermore, $T$ satisfies
    \begin{equation*}
        \norm{\dive T}_\infty = \abs{x^+ - x^-} \norm{\varphi'}_\infty \le \frac{2 \max_{\S^{n - 1}} \phi}{M (1 - M)} = C_4. 
    \end{equation*}
    \begin{enumerate}
        \item 
            The first estimate follows from
            \begin{align*}
                P_{\T^n}^\phi (L) - P_{\T^n}^\phi(F)
                &= \int_{\partial L} \phi(\nu_L) \de \Ha^{n - 1}
                - \int_{\reb F} \phi(\nu_F) \de \Ha^{n - 1} \\
                &\le \int_{\partial L} T \cdot \nu_L \de \Ha^{n - 1}
                - \int_{\reb F} T \cdot \nu_F \de \Ha^{n - 1} \\
                &= \int_L \dive T - \int_F \dive T
                \le \norm{\dive T}_\infty \abs{F \Delta L} \le C_4 \abs{F \Delta L}.
            \end{align*}

        \item 
            For the second inequality we note that
            \begin{align*}
                \reb(L \cap F) 
                &= \rb{F^{(1)} \cap \partial L} \cup \rb{L \cap \reb F}
                \cup \cb{x \in \reb F \cap \partial L : \nu_F(x) = \nu_L(x)}, \\
                \reb(F \setminus L) 
                &= \rb{F^{(1)} \cap \partial L} \cup \rb{\reb F \setminus \overline L}
                \cup \cb{x \in \reb F \cap \partial L : \nu_F(x) = - \nu_L(x)},
            \end{align*}
            where $F^{(1)}$ is the set of points of density $1$ of $F$ (see for instance \cite[Theorem 16.3]{Mag12}).
            This implies
            \begin{align*}
                P_{\T^n}^\phi (L \cap F) - P_{\T^n}^\phi(F)
                &= \int_{\reb (L \cap F)} \phi(\nu_{L \cap F}) \de \Ha^{n - 1}
                - \int_{\reb F} \phi(\nu_F) \de \Ha^{n - 1} \\
                &= \int_{F^{(1)} \cap \partial L} \phi(\nu_L) \de \Ha^{n - 1}
                - \int_{\rb{\reb F \setminus \overline L} \cup \cb{\nu_F = - \nu_L}} \phi(\nu_F) \de \Ha^{n - 1} \\
                &\le \int_{F^{(1)} \cap \partial L} T \cdot \nu_L \de \Ha^{n - 1}
                - \int_{\rb{\reb F \setminus \overline L} \cup \cb{\nu_F = - \nu_L}} T \cdot \nu_F \de \Ha^{n - 1} \\
                &= - \int_{\reb (F \setminus L)} T \cdot \nu_{F \setminus L} \de \Ha^{n - 1}
                = - \int_{F \setminus L} \dive T \le C_4 \abs{F \setminus L}.
            \end{align*}

        \item 
            For the third estimate we may argue similarly as in the previous point.
            Alternatively, setting $\tilde \phi(x) \defeq \phi(- x)$ for any $x \in \R^n$, it holds
            \begin{equation*}
                P_{\T^n}^\phi (L \cup F) - P_{\T^n}^\phi(F) 
                = P_{\T^n}^{\tilde \phi} (L^c \cap F^c) - P_{\T^n}^{\tilde \phi}(F^c) 
                \le C_4 \abs{F^c \setminus L^c} = C_4 \abs{L \setminus F}.
                \qedhere
            \end{equation*}
            
    \end{enumerate}
    
\end{proof}

We are now ready to prove \cref{thm:2.strictminL1}.

\begin{proof}[Proof of \cref{thm:2.strictminL1}]

    By \cref{thm:4.C1infty} there exists $\delta > 0$ such that
    \begin{equation} \label{eq:5.C1min}
        J^\phi(g_1, g_2) 
        \ge J^\phi(L) + \frac{C_L^\phi}{4} \rb{\sum_{i = 1}^2 \norm{\nabla g_i}_{L^2(\T^{n - 1})}^2}
    \end{equation}
    for any strip $(g_1, g_2) \in \St_M$ with $d_{C^1}((g_1, g_2), L) < \delta$.
    We assume by contradiction that there exists a sequence $\cb{E_h}_{h \in \N}$ 
    of sets of finite perimeter in $\T^n$, with $\abs{E_h} = M$ for any $h \in \N$,
    such that $E_h \to L = (\ell_1, \ell_2)$ and
    \begin{equation*}
        J^\phi(E_h) < J^\phi(L) + \frac{C_L^\phi}{16 (n - 1)} \big( d_{L^1}(E_h, L) \big)^2, \qquad \text{for any } h \in \N.
    \end{equation*}
    Up to translating the sets, we may assume that 
    $\1_{E_h} \to \1_L$ in $L^1(\T^n)$ as $h \to \infty$.
    For any $h \in \N$ let $F_h \subset \T^n$ be a minimizer for the penalized functional
    \begin{equation*}
        J_h^\phi(F) = J^\phi(F) + \Lambda_1 d_{L^1}(F, E_h) + \Lambda_2 \big| \abs{F} - M \big|, \qquad \qquad F \subset \T^n,
    \end{equation*}
    where the constants $\Lambda_1, \Lambda_2 > 0$ will be chosen later.
    Note that the existence of such a minimizer follows as in \cref{rem:2.minexst}.
    Up to a subsequence, we may also assume that $F_h \to \tilde F$, where $\tilde F \subset \T^n$ is a minimizer of 
    \begin{equation*}
        J^\phi(F) + \Lambda_1 d_{L^1}(F, L) + \Lambda_2 \big| \abs{F} - M \big|, \qquad \qquad F \subset \T^n.
    \end{equation*}
    Once again up to translation, we may assume that $\tilde F$ is a minimizer of
    \begin{equation*}
        \tilde J^\phi(F) = J^\phi(F) + \Lambda_1 \abs{F \Delta L} + \Lambda_2 \big| \abs{F} - M \big|, \qquad \qquad F \subset \T^n.
    \end{equation*}

    \vspace{5 pt}
    \noindent
    \textit{Step one}:
    we prove that choosing $\Lambda_1$ sufficiently large one has $\tilde F = L$.
    Indeed, by means of \cref{lem:2.nonlocont} and \cref{lem:4.boundper} we have
    \begin{align*}
        0 &\ge \tilde J^\phi(\tilde F) - \tilde J^\phi(L) = J^\phi(\tilde F) + \Lambda_1 \abs{\tilde F \Delta L} 
        + \Lambda_2 \abs{|\tilde F| - M} - J^\phi(L) \\
        &= P_{\T^n}^\phi(\tilde F) - P_{\T^n}^\phi(L) + \gamma \rb{\F(\tilde F) - \F(L)} 
        + \Lambda_1 \abs{\tilde F \Delta L} + \Lambda_2 \abs{|\tilde F| - M} \\
        &\ge - C_4 \abs{\tilde F \Delta L} - \gamma C_3 \abs{\tilde F \Delta L}
        + \Lambda_1 \abs{\tilde F \Delta L} 
        = \rb{\Lambda_1 - \gamma C_3 - C_4} \abs{\tilde F \Delta L},
    \end{align*}
    so that by choosing $\Lambda_1 > \gamma C_3 + C_4$ we conclude that $\tilde F = L = (\ell_1, \ell_2)$.

    \vspace{5 pt}
    \noindent
    \textit{Step two}:
    arguing in a similar way, we show that for some $\Lambda_2$ sufficiently large 
    it holds $\abs{F_h} = M$ for any $h \in \N$.
    To see that, assume by contradiction that $\abs{F_h} \not = M$,
    and set, for $\abs{\varepsilon}$ sufficiently small,
    $L^\varepsilon \defeq (\ell_1 - \varepsilon, \ell_2 + \varepsilon)$.
    Then define, for $h \in \N$ sufficiently large,
    \begin{equation*}
        F_h' \defeq 
        \begin{cases}
            F_h \cup L^{\varepsilon_h} \qquad &\text{ if } \abs{F_h} < M, \\
            F_h \cap L^{\varepsilon_h} \qquad &\text{ if } \abs{F_h} > M,
        \end{cases}
    \end{equation*}
    where $\varepsilon_h \in \R$ is chosen so that $\abs{F_h'} = M$.
    This is possible by continuity of the Lebesgue measure, since $F_h \to L$.
    Note that we have either $F_h \subset F_h'$ or $F_h' \subset F_h$,
    therefore $\abs{F_h' \Delta F_h} = \abs{\abs{F_h} - M}$.
    Moreover, $\varepsilon_h \to 0$ as $h \to \infty$, so that by \cref{lem:4.boundper} one has
    \begin{equation*}
        P_{\T^n}^\phi (F_h) - P_{\T^n}^\phi(F_h') \ge - 2 C_4(\phi, M) \big| F_h' \Delta F_h \big|
    \end{equation*}
    for $h \in \N$ sufficiently large.
    Again by \cref{lem:2.nonlocont} we infer that
    \begin{align*}
        0 &\ge J_h^\phi(F_h) - J_h^\phi(F_h') 
        = J^\phi(F_h) - J^\phi(F_h') + \Lambda_1 \big( d_{L^1}(F_h, E_h) - d_{L^1}(F_h', E_h) \big) + \Lambda_2 \Big| \abs{F_h} - M \Big| \\
        &\ge P_{\T^n}^\phi (F_h) - P_{\T^n}^\phi(F_h') + \gamma \Big( \F(F_h) - \F(F_h') \Big)
        - \Lambda_1 d_{L^1}(F_h, F_h') + \Lambda_2 \Big| \abs{F_h} - M \Big| \\
        &\ge (- 2 C_4 - \gamma C_3 - \Lambda_1) \big| F_h' \Delta F_h \big| + \Lambda_2 \Big| \abs{F_h} - M \Big|
        = (\Lambda_2 - \gamma C_3 - 2 C_4 - \Lambda_1) \Big| \abs{F_h} - M \Big|.
    \end{align*}
    Choosing $\Lambda_2 > \gamma C_3 + 2 C_4 + \Lambda_1$ we get a contradiction with the minimality of $F_h$.
   
    \vspace{5 pt}
    \noindent
    \textit{Step three}:
    we claim that the periodic extension $(F_h)_{\R^n}$ of $F_h$ to $\R^n$ is a $(\phi, \Lambda, 1/2)$-perimeter minimizer for any $h \in \N$,
    where $\Lambda = \gamma C_3 + \Lambda_1 + \Lambda_2$.
    Indeed, let $x \in \R^n$, $0 < r < 1/2$, and $F \subset \R^n$ be a set of locally finite perimeter with 
    $(F_h)_{\R^n} \Delta F \subset \subset B(x, r)$.
    Up to a translation of $(F_h)_{\R^n}$  and $F$ we may assume that $B(x, r) \subset (0, 1)^n$, then
    \begin{multline*}
        P^\phi \big( (F_h)_{\R^n}; B(x, r) \big) - P^\phi \big( F; B(x, r) \big)
        = P_{\T^n}^\phi(F_h) - P_{\T^n}^\phi(F) \\
        \begin{aligned}
            &= J_h^\phi(F_h) - J_h^\phi(F) - \gamma \Big( \F(F_h) - \F(F) \Big)
            - \Lambda_1 \big( d_{L^1}(F_h, E_h) - d_{L^1}(F, E_h) \big) + \Lambda_2 \Big| \abs{F} - M \Big| \\
            &\le \gamma C_3 \big| F_h \Delta F \big| 
            + \Lambda_1 \big| F_h \Delta F \big| + \Lambda_2 \big| F_h \Delta F \big|
            = (\gamma C_3 + \Lambda_1 + \Lambda_2) \big| F_h \Delta F \big|.
        \end{aligned}
    \end{multline*}
    We may apply \cref{thm:5.almostreg} to conclude that, for $h \in \N$ large enough, 
    $F_h = (g_{1, h}, g_{2, h}) \in \St_M$ for some functions $g_{1, h}, g_{2, h} \in C^1(\T^{n - 1})$ with
    \begin{equation*}
        \lim_{h \to \infty} \norm{g_{i, h} - \ell_i}_\infty = 0, \qquad
        \lim_{h \to \infty} \norm{\nabla g_{i, h}}_\infty = 0, \qquad \qquad i = 1, 2,
    \end{equation*}
    where $L = (\ell_1, \ell_2)$, so that $d_{C^1}(F_h, L) \to 0$ as $h \to \infty$.
    We may also assume that, up to a subsequence and a translation, for any $h \in \N$ the set $F_h$ is a strip
    and $\int_{\T^{n - 1}} g_{i, h} = \ell_i$ for $i = 1, 2$.

    \vspace{5 pt}
    \noindent
    \textit{Step four}:
    we claim that
    \begin{equation*}
        \lim_{h \to \infty} \dfrac{d_{L^1}(F_h,L)}{d_{L^1}(E_h, L)} = 1.
    \end{equation*}
    Indeed, by minimality of $F_h$ for $J_h^\phi$ and by \eqref{eq:5.C1min}, we have
    \begin{multline*}
        J^\phi (F_h) + \Lambda_1 d_{L^1}(F_h, E_h) = J_h^\phi(F_h) \le J_h^\phi(E_h) = J^\phi(E_h) \\
        < J^\phi(L) + \dfrac{C_L^\phi}{16 (n - 1)} \big( d_{L^1}(E_h, L) \big)^2 
        \le J^\phi(F_h) + \dfrac{C_L^\phi}{16 (n - 1)} \big( d_{L^1}(E_h, L) \big)^2.
    \end{multline*}
    As a consequence,
    \begin{equation*}
        \lim_{h \to \infty} \abs{\dfrac{d_{L^1}(E_h, L) - d_{L^1}(F_h, L)}{d_{L^1}(E_h, L)}} 
        \le \lim_{h \to \infty} \dfrac{d_{L^1}(F_h, E_h)}{d_{L^1}(E_h, L)} 
        \le \dfrac{C_L^\phi}{16 \Lambda_1 (n - 1)} \lim_{h \to \infty} d_{L^1}(E_h, L) = 0,
    \end{equation*}
    proving our claim.
    Recalling that $F_h = (g_{1, h}, g_{2, h})$, we conclude by noting that
    \begin{align*}
        J^\phi(F_h) &\le J^\phi(E_h) < J^\phi(L) + \frac{C_L^\phi}{16 (n - 1)} \big( d_{L^1}(E_h, L) \big)^2 \\
        &\le J^\phi(L) + \frac{C_L^\phi}{8 (n - 1)} \big( d_{L^1}(F_h, L) \big)^2
        \le J^\phi(L) + \frac{C_L^\phi}{8 (n - 1)} \rb{\sum_{i = 1}^2 \int_{\T^{n - 1}} \abs{g_{i, h} - \ell_i}}^2 \\
        &\le J^\phi(L) + \frac{C_L^\phi}{4 (n - 1)} \rb{\sum_{i = 1}^2 \int_{\T^{n - 1}} \abs{g_{i, h} - \ell_i}^2}
        \le J^\phi(L) + \frac{C_L^\phi}{4} \rb{\sum_{i = 1}^2 \int_{\T^{n - 1}} \abs{\nabla g_{i, h}}^2}
    \end{align*}
    for $h \in \N$ sufficiently large.
    Thus we reached a contradiction with \eqref{eq:5.C1min}.
    \qedhere

\end{proof}

One can easily prove the following by means 
of \cref{prop:3.lamstab}, \cref{thm:2.strictminL1} and \cref{rem:3.weakuniell},
which in turn implies \cref{cor:2.lamminL1}.

\begin{cor} \label{cor:5.lamminL1}

    Let $L$ be the lamella of volume $M \in (0, 1)$, and let $\phi : \R^n \to [0, + \infty)$ 
    be a uniformly elliptic surface tension satisfying \eqref{eq:3.weakuniell} for some $\lambda > 0$.
    Then, setting 
    \begin{equation*}
        \gamma_0 \defeq \dfrac{\lambda}{4 (M - M^2) (n - 1)},
    \end{equation*}
    $L$ is an \textit{isolated $L^1$-local minimizer} for $J_\gamma^\phi$ for any $\gamma \in [0, \gamma_0)$;
    more precisely, fixing $\gamma \in [0, \gamma_0)$, there exists $\delta > 0$ such that
    \begin{equation*}
        J_\gamma^\phi(E) \ge J_\gamma^\phi(L) + \frac{C_L^\phi}{16 (n - 1)} \big( d_{L^1}(E, L) \big)^2
    \end{equation*}
    for any set $E \subset \T^n$ with $\abs{E} = M$ and $d_{L^1}(E, L) < \delta$, where
    \begin{equation*}
         C_L^\phi \defeq \inf \Big\{ \partial^2 J^\phi(L) [\psi_1, \psi_2] : 
        \psi_1, \psi_2 \in V, \ \norm{(\psi_1, \psi_2)}_{V \times V} = 1 \Big\} > 0.
    \end{equation*}
    
\end{cor}

\section{The horizontally flat case} \label{sec:6}

The results obtained in the previous section required some strong regularity assumptions 
on the surface tension $\phi$ for the second variation formula to make sense.
Here we will consider the case of a horizontally flat surface tension $\phi$,
as to prove \cref{thm:2.lamminL1hor}.

From a geometrical point of view, the proof is obtained by inner approximation of the Wulff shape for a horizontally flat $\phi$
through strictly convex sets of class $C^2$ tangent to $W_\phi$ at its horizontal faces. 
By the assumption on the singularity of $\phi$ at $\pm e_n$, we may ensure that the surface tension associated with an approximating set 
satisfies \eqref{eq:3.weakuniell} for arbitrarily large values of $\lambda$.

\begin{proof}[Proof of \cref{thm:2.lamminL1hor}]

    We divide the proof into four steps. 
    In Step one we prove the result for a surface tension $\phi$
    whose Wulff shape is a right circular cylinder symmetric with respect to the origin.
    This is the heart of the proof, and it is achieved by an approximation argument employed in \cite{Bon13}.
    In Step two we remove the assumption that the cylinder is right circular, assuming it to be oblique instead.
    In Step three we prove the assertion whenever $\partial W_\phi$ contains the horizontal faces 
    of an oblique cylinder symmetric with respect to the origin.
    Finally, in Step four we no longer require symmetry.
    
    \vspace{5 pt}
    \noindent
    \textit{Step one}:
    consider the open right circular cylinder $C \subset \R^n$ centered at the origin, given by
    \begin{equation} \label{eq:5.cyl}
        C = \Big\{ (x', x_n) \in \R^n : \abs{x'} < a, \ \abs{x_n} < b \Big\},
    \end{equation}
    where $a > 0$ is the radius, and $b > 0$ is half the height (see \cref{fig:6.horiflat1})
    By means of \cref{prop:2.surfchar} its surface tension is easily seen to be
    \begin{equation*}
        \phi(x', x_n) = a \abs{x'} + b \abs{x_n}, \qquad \qquad (x', x_n) \in \R^n.
    \end{equation*}
    For any $\varepsilon > 0$ sufficiently small we look for a uniformly elliptic surface tension $\phi^\varepsilon$ 
    satisfying the following properties:
    \begin{enumerate}
        \item 
            $\phi^\varepsilon (\pm e_n) = \phi (\pm e_n)$;

        \item 
            $\phi^\varepsilon \le \phi$;

        \item 
            $\phi^\varepsilon$ is uniformly elliptic and 
            $\nabla_{x'}^2 \phi^\varepsilon (\pm e_n) \sb{v', v'} \ge \frac{\lambda}{\varepsilon} \abs{v'}^2$ for any $v' \in \R^{n - 1}$
            for some constant $\lambda > 0$ not depending on $\varepsilon$.
            
    \end{enumerate}
    In that case, by (i) it follows that $P_{\T^n}^{\phi^\varepsilon}(L) = P_{\T^n}^\phi(L)$,
    whereas by (ii) it must be $P_{\T^n}^{\phi^\varepsilon}(E) \le P_{\T^n}^\phi(E)$
    for any set of finite perimeter $E \subset \T^n$.
    Finally, (iii) and \cref{cor:5.lamminL1} ensure that $L$ is an isolated $L^1$-local minimizer of $J^{\phi^\varepsilon}_\gamma$ 
    for any $\gamma < \gamma_0$, where
    \begin{equation*}
        \gamma_0 = \frac{\lambda}{4 \varepsilon (M - M^2) (n - 1)}.
    \end{equation*}
    This means that, for any $\gamma < \gamma_0$, there exists $\delta > 0$ such that
    \begin{equation*} 
        J_\gamma^{\phi^\varepsilon}(L) < J_\gamma^{\phi^\varepsilon}(E)
    \end{equation*}
    for any $E \subset \T^n$ with $\abs{E} = M$ and $0 < d_{L^1}(E, L) < \delta$, so that
    \begin{align*}
        J_\gamma^\phi(L) = J_\gamma^{\phi^\varepsilon}(L) 
        < J_\gamma^{\phi^\varepsilon}(E) \le J_\gamma^\phi(E).
    \end{align*}
    By arbitrariness of $\varepsilon$ and since $\gamma_0 \to + \infty$ as $\varepsilon \to 0^+$ 
    we may conclude that $L$ is an isolated $L^1$-local minimizer of $J_\gamma^\phi$ for any $\gamma \ge 0$.
    The surface tension $\phi^\varepsilon \in C^2(\R^n \setminus \cb{0})$ given by
    \begin{equation} \label{eq:5.cylapprox}
        \phi^\varepsilon (x', x_n) 
        \defeq \frac{a - b \varepsilon}{1 - \varepsilon^2} \sqrt{\abs{x'}^2 + \varepsilon^2 \abs{x_n}^2} 
        + \frac{b - a \varepsilon}{1 - \varepsilon^2} \sqrt{\varepsilon^2 \abs{x'}^2 + \abs{x_n}^2}, 
        \qquad \qquad (x', x_n) \in \R^n,
    \end{equation}
    where $\varepsilon > 0$ is a fixed constant smaller than $1$, $a/b$ and $b/a$, satisfies all three properties:
    clearly (i) holds since $\phi^\varepsilon (\pm e_n) = b = \phi (\pm e_n)$, moreover
    \begin{equation*}
        \phi^\varepsilon (x', x_n)
        \le \frac{a - b \varepsilon}{1 - \varepsilon^2} \big( \abs{x'} + \varepsilon \abs{x_n} \big)
        + \frac{b - a \varepsilon}{1 - \varepsilon^2} \big( \varepsilon \abs{x'} + \abs{x_n} \big)
        \le a \abs{x'} + b \abs{x_n} = \phi(x', x_n),
    \end{equation*}
    for any $(x', x_n) \in \R^n$, so that we also have (ii).
    Finally, a straightforward computation shows that, 
    for $x, v \in \S^{n - 1}$ with $x \perp v$,
    \begin{multline*}
        \nabla^2 \phi^\varepsilon (x) \sb{v, v} 
        = \dfrac{a - b \varepsilon}{\rb{\sqrt{\abs{x'}^2 + \varepsilon^2 x_n^2}}^3} 
        \rb{\dfrac{\varepsilon^2}{1 - \varepsilon^2}
        + \rb{\abs{x'}^2 \abs{v'}^2 - (x' \cdot v')^2}} \\
        + \dfrac{b - a \varepsilon}{\rb{\sqrt{\varepsilon^2 \abs{x'}^2 + x_n^2}}^3} 
        \rb{\dfrac{\varepsilon^2}{1 - \varepsilon^2}
        - \varepsilon^2 \rb{\abs{x'}^2 \abs{v'}^2 - (x' \cdot v')^2}}. 
    \end{multline*}
    This implies that $\phi^\varepsilon$ is uniformly elliptic with
    \begin{equation} \label{eq:6.uniellconst}
        \nabla^2 \phi^\varepsilon (x) \sb{v, v} 
        \ge \dfrac{(a + b) \varepsilon^2}{(1 + \varepsilon)}
        \qquad \qquad x, v \in \S^{n - 1}, \ x \perp v,
    \end{equation}
    and moreover
    \begin{equation*}
        \nabla_{x'}^2 \phi^\varepsilon (\pm e_n) \sb{v', v'} 
        = \rb{\dfrac{a (1 + \varepsilon^2)}{\varepsilon} - b} \abs{v'}^2 \ge \dfrac{a - b \varepsilon}{\varepsilon} \abs{v'}^2,
        \qquad \qquad v' \in \R^{n - 1},
    \end{equation*}
    yielding (iii) as well as the conclusion of the first step of the proof.

    \begin{figure}
    \centering
    \begin{tikzpicture}[scale = 0.8]
        \draw[dashed, fill = gray!50] (- 0.5, - 1.5) -- (0.5, - 1.5) -- (0.5, 1.5) -- (- 0.5, 1.5) -- cycle;
        \draw[thick, -latex](0, - 2.5) -- (0, 2.5) node[left]{$x_2$};
        \draw[thick, -latex](- 3, 0) -- (3, 0) node[above]{$x_1$};
        \draw[fill = black] (0.5, 0) circle (0.05) node[anchor = north west]{$a$};
        \draw[fill = black] (- 0.5, 0) circle (0.05) node[anchor = north east]{$- a$};
        \draw[fill = black] (0, 1.5) circle (0.05) node[anchor = south east]{$b$};
        \draw[fill = black] (0, - 1.5) circle (0.05) node[anchor = north east]{$- b$};
        \draw(1, 1.5) node{$C$};
    \end{tikzpicture}
    \hspace{1 cm}
    \begin{tikzpicture} [scale = 0.8]
        \draw[dashed, fill = gray!50] (- 1.5, - 1.5) -- (- 0.5, - 1.5) -- (1.5, 1.5) -- (0.5, 1.5) -- cycle;
        \draw[thick, -latex](0, - 2.5) -- (0, 2.5) node[left]{$x_2$};
        \draw[thick, -latex](- 3, 0) -- (3, 0) node[above]{$x_1$};
        \draw[fill = black] (0.5, 0) circle (0.05) node[anchor = north west]{$a$};
        \draw[fill = black] (- 0.5, 0) circle (0.05) node[anchor = south east]{$- a$};
        \draw[fill = black] (0, 1.5) circle (0.05) node[anchor = south east]{$b$};
        \draw[fill = black] (0, - 1.5) circle (0.05) node[anchor = north east]{$- b$};
        \draw[thick, ->] (0, 1.5) -- (1, 1.5) node[above]{\scriptsize{$(\tau', 0)$}};
        \draw(1.6, 0.7) node{$C'$};
    \end{tikzpicture}
    \caption{The Wulff shapes for Step one and Step two of the proof of \cref{thm:2.lamminL1hor}.
    If the projections of the two horizontal faces on the hyperplane $\cb{x_n = 0}$ had nonempty intersection 
    then one could skip Step two altogether and get the same conclusion by Step three and Step four.}
    \label{fig:6.horiflat1}
\end{figure}
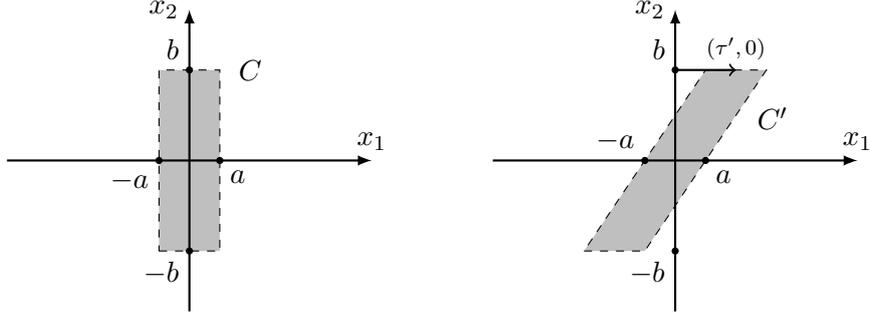

    \vspace{5 pt}
    \noindent
    \textit{Step two}:
    consider the surface tension $\phi$ associated with an open oblique circular cylinder 
    $C' \subset \R^n$ with radius $a$, and height $2 b$ symmetric with respect to the origin (see \cref{fig:6.horiflat1}).
    Denoting by $(\tau', 0)$ the vector joining $(0, b)$ with the center of the upper face of $C'$,
    we consider the linear map associated with the matrix
    \begin{equation*}
        A = 
        \begin{pmatrix}
            1 & 0 & \cdots & 0 & \frac{\tau_1}{b} \\
            0 & 1 & \cdots & 0 & \frac{\tau_2}{b} \\
            \vdots & \ddots & \ddots & \vdots & \vdots \\
            0 & 0 & \cdots & 1 & \frac{\tau_{n - 1}}{b} \\
            0 & 0 & \cdots & 0 & 1 \\
        \end{pmatrix}
        = \Id + \dfrac{(\tau', 0)}{b} \otimes e_n.
    \end{equation*}
    Notice that this gives a linear bijection between $C$ and $C'$, 
    where $C$ is defined as in \eqref{eq:5.cyl}, whose surface tension is now denoted by $\phi_C$.
    Again by \cref{prop:2.surfchar} we know that the relation between $\phi$ and $\phi_C$ is
    \begin{equation*}
        \phi(x) = \phi_C (A^* x) = a \abs{x'} + \abs{x' \cdot \tau' + b x_n}, 
        \qquad \qquad x = (x', x_n) \in \R^n.
    \end{equation*}
    The surface tension $\phi^\varepsilon \in C^2(\R^n \setminus \cb{0})$ defined by
    \begin{equation*}
        \phi^\varepsilon (x) \defeq \phi_C^\varepsilon (A^* x), \qquad \qquad x \in \R^n,
    \end{equation*}
    where $\phi_C^\varepsilon$ denotes the surface tension defined in \eqref{eq:5.cylapprox},
    once again satisfies (i), (ii) and (iii), allowing to conclude:
    clearly $\phi^\varepsilon(\pm e_n) = b = \phi (\pm e_n)$ and
    \begin{equation*}
        \phi^\varepsilon(x) = \phi_C^\varepsilon(A^* x) \le \phi_C(A^* x) = \phi(x),
    \end{equation*}
    so that only (iii) is left to prove.
    Notice that $\phi^\varepsilon$ is uniformly elliptic:
    by means of \eqref{eq:6.uniellconst} 
    it suffices to prove that there exists $\mu > 0$ such that, 
    for any $x, v \in \S^{n - 1}$ with $x \perp v$, it holds 
    \begin{equation} \label{eq:6.uniellcyl}
        \begin{split}
            \nabla^2 \phi^\varepsilon (x) \sb{v, v} &= \nabla^2 \phi_C^\varepsilon (A^*x) \sb{A^* v, A^* v} \\
            &\ge \dfrac{(a + b) \varepsilon^2}{(1 + \varepsilon) \abs{A^* x}^3} 
            \rb{\abs{A^* v}^2 \abs{A^* x}^2 - (A^* x \cdot A^* v)^2} \ge \mu.
        \end{split}
    \end{equation} 
    Indeed, if by contradiction we had
    \begin{equation*}
        \inf_{\substack{x, v \in \S^{n - 1} \\ x \perp v}} 
        \rb{\abs{A^* v}^2 \abs{A^* x}^2 - (A^* x \cdot A^* v)^2} = 0,
    \end{equation*}
    by compactness of $\S^{n - 1}$ there would exist $\bar x, \bar v \in \S^{n - 1}$ with $\bar x \perp \bar v$
    such that
    \begin{equation*}
        \abs{A^* v}^2 \abs{A^* x}^2 = (A^* x \cdot A^* v)^2,
    \end{equation*}
    which in turn implies that $A^* x$ and $A^* v$ are linearly dependent.
    Since $A$ is invertible the same is true for $x$ and $v$, against the assumption of orthogonality,
    proving the validity of \eqref{eq:6.uniellcyl}.
    Since $\nabla^2 \phi_C^\varepsilon(x) x = 0$ for any $x \ne 0$, by the very same formula we note that
    \begin{align*}
        \nabla_{x'}^2 \phi^\varepsilon (\pm e_n) \sb{v', v'} 
        &= \nabla^2 \phi^\varepsilon (\pm e_n) \sb{(v', 0), (v', 0)}
        = \nabla^2 \phi_C^\varepsilon (\pm A^* e_n) \sb{A^*(v', 0), A^*(v', 0)} \\
        &= \nabla^2 \phi_C^\varepsilon (\pm e_n) 
        \sb{\rb{v', \frac{\tau' \cdot v'}{b}}, \rb{v', \frac{\tau' \cdot v'}{b}}} \\
        &= \nabla_{x'}^2 \phi_C^\varepsilon (\pm e_n) \sb{v', v'}
        \ge \dfrac{a - b \varepsilon}{\varepsilon} \abs{v'}^2
    \end{align*}
    for any $v \in \R^{n - 1}$, concluding the proof of Step two.
    
    \vspace{5 pt}
    \noindent
    \textit{Step three}:
    let $\phi$ be a surface tension and assume there exist $a, b > 0$, $\tau' \in \R^{n - 1}$
    with $C'$ defined as in the previous step such that $C' \subset W_\phi$ and
    $\cb{x \in \reb C' : \nu_{C'}(x) = \pm e_n} \subset \reb W_\phi$ (see \cref{fig:6.horiflat2}).
    Therefore, denoting by $\phi_{C'}$ the surface tension whose Wulff shape is $C'$, one gets
    \begin{equation*}
        \phi_{C'}(x) = \sup \Big\{ x \cdot y : y \in C' \Big\} 
        \le \sup \Big\{ x \cdot y : y \in W_\phi \Big\} = \phi(x), \qquad \qquad x \in \R^n,
    \end{equation*}
    as well as
    \begin{equation*}
        \phi(\pm e_n) = \sup \Big\{ \pm y_n : (y', y_n) \in W_\phi \Big\}
        = \sup \Big\{ \pm y_n : (y', y_n) \in C' \Big\} = \phi_{C'}(\pm e_n).
    \end{equation*}
    Hence $P_{\T^n}^{\phi_{C'}} (L) = P_{\T^n}^\phi (L)$, 
    and also $P_{\T^n}^{\phi_{C'}} (E) \le P_{\T^n}^\phi (E)$ for any $E \subset \T^n$.
    By Step two we conclude.

    \begin{figure}
    \centering
    \begin{tikzpicture} [scale = 0.8]
        \draw[dashed, fill = gray!50] (- 1.5, - 1.5) -- (- 0.5, - 1.5) -- (0.7, - 1) -- (1.7, 0) -- (1.5, 1.5) -- (0.5, 1.5) 
        -- (- 0.75, 1.25) -- (- 2, 0) to[out = - 120, in = 160]  cycle;
        \draw[densely dotted] (- 1.5, - 1.5) -- (- 0.5, - 1.5) -- (1.5, 1.5) -- (0.5, 1.5) -- cycle;
        \draw[thick, -latex](0, - 2.5) -- (0, 2.5) node[left]{$x_2$};
        \draw[thick, -latex](- 3, 0) -- (3, 0) node[above]{$x_1$};
        \draw[fill = black] (0.5, 0) circle (0.05) node[anchor = north west]{$a$};
        \draw[fill = black] (- 0.5, 0) circle (0.05) node[anchor = south east]{$- a$};
        \draw[fill = black] (0, 1.5) circle (0.05) node[anchor = south east]{$b$};
        \draw[fill = black] (0, - 1.5) circle (0.05) node[anchor = north east]{$- b$};
        \draw(1.6, - 0.7) node{$W_\phi$};
    \end{tikzpicture}
    \hspace{1 cm}
    \begin{tikzpicture} [scale = 0.8]
        \draw[dashed, fill = gray!50] (- 1.5, - 1.5) -- (- 0.5, - 1.5) -- (0.7, - 1) -- (1.7, 0) -- (1.5, 1.5) -- (0.5, 1.5) 
        -- (- 0.75, 1.25) -- (- 2, 0) to[out = - 120, in = 160]  cycle;
        \draw[densely dotted] (- 1.5, - 1.5) -- (- 0.5, - 1.5) -- (1.5, 1.5) -- (0.5, 1.5) -- cycle;
        \draw[thick, -latex](- 0.5, - 2) -- (- 0.5, 3) node[left]{$x_2$};
        \draw[thick, -latex](- 3.5, 0.5) -- (2.5, 0.5) node[above]{$x_1$};
        \draw[thick, ->] (- 0.5, 0.5) -- (0, 0) node[above]{$z$};
        \draw(1.6, - 0.7) node{$W_\phi$};
    \end{tikzpicture}
    \caption{The Wulff shapes for Step three and Step four of the proof of \cref{thm:2.lamminL1hor}.}
    \label{fig:6.horiflat2}
\end{figure}
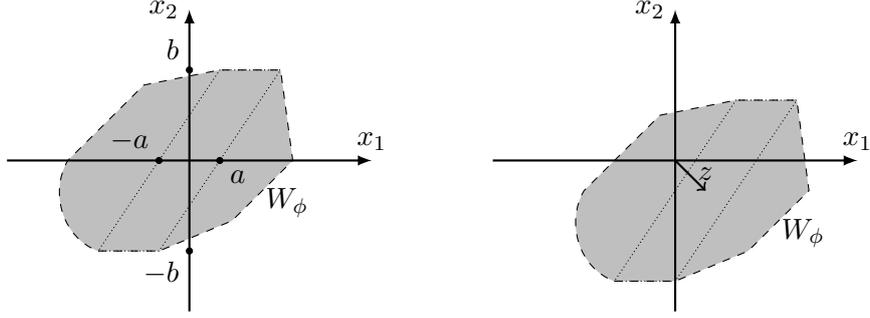
    
    \vspace{5 pt}
    \noindent
    \textit{Step four}:
    let $\phi$ be a horizontally flat surface tension. 
    By definition $W_\phi$ has upper and lower horizontal faces, and up to a translation of $W_\phi$
    we may find an oblique cylinder $C'$ satisfying the geometric condition of Step three (see \cref{fig:6.horiflat2}).
    Considering the vector $z \in \R^n$ joining $0 \in \R^n$ with the center of $C'$, 
    we see that $z$ belongs to the Wulff shape $W_\phi$ of $\phi$.
    In particular $0 \in W_\phi - z$, and \cref{prop:2.surfchar} shows that 
    there exists a surface tension $\hat \phi$ whose Wulff shape is $W_\phi - z$, 
    defined by
    \begin{align*}
        \hat \phi(x) &= \sup \Big\{ x \cdot y : y \in W_\phi - z \Big\} 
        = \sup \Big\{ x \cdot (y - z) : y \in W_\phi \Big\} \\
        &= \sup \Big\{ x \cdot y : y \in W_\phi \Big\} - x \cdot z = \phi(x) - x \cdot z,
        \qquad \qquad x \in \R^n.
    \end{align*}
    Given $E \subset \T^n$, the Generalized Divergence Theorem yields
    \begin{equation*}
        P_{\T^n}^\phi (E) 
        = P_{\T^n}^{\hat \phi}(E) + \int_{\reb E} \nu_E(x) \cdot z \de \Ha^{n - 1}
        = P_{\T^n}^{\hat \phi}(E),
    \end{equation*}
    and we conclude the proof, on noting that $\hat \phi$ satisfies the assumptions of Step three. \qedhere
    
\end{proof}

\section{Global minimality of the lamella} \label{sec:7}

    In this final section we aim to to prove some results about global minimality of the lamella, as to prove \cref{thm:2.globmin2}.
    We follow the arguments of \cite[Section 5]{AFM13}, first proving that if the surface tension is uniformly elliptic (or also horizontally flat)
    and the lamella is the unique minimizer of the periodic anisotropic isoperimetric problem \eqref{eq:2.wulffper}, 
    then it is also a minimizer for problem \eqref{eq:1.anOKprob} for sufficiently small values of $\gamma$,
    which is the content of \cref{prop:7.stillmin}.
    Then, we show in \cref{thm:7.wulffminell} that for $n = 2$, under the assumption \eqref{eq:2.minilam} on the uniformly elliptic anisotropy $\phi$,
    the horizontal lamella is the unique minimizer for \eqref{eq:2.wulffper}
    for $M$ sufficiently close to $1/2$ with explicit minimality range, yielding \cref{thm:2.globmin2}. 

    \begin{prop} \label{prop:7.stillmin}

        Let $L$ be the lamella of volume $M \in (0, 1)$, and let the surface tension $\phi$ be either uniformly elliptic or horizontally flat.
        If $L$ is the unique (up to translations) global minimizer of the periodic Wulff problem \eqref{eq:2.wulffper}, 
        then it is also the unique global minimizer for $J_\gamma^\phi$ as long as $\gamma > 0$ is sufficiently small.
        
    \end{prop}

    \begin{proof} 

        Assume by contradiction that there exists a decreasing sequence 
        $\cb{\gamma_h}_{h \in \N}$ with $\gamma_h \to 0$ as $h \to \infty$
        and a sequence $\cb{E_h}_{h \in \N}$ of sets of finite perimeter in $\T^n$ 
        with $\abs{E_h} = M$ and $d_{L^1} (E_h, L) > 0$ for any $h \in \N$ such that 
        \begin{equation*}
            J_{\gamma_h}^\phi(E_h) = \inf \cb{J_{\gamma_h}^\phi(E) : \abs{E} = M}.
        \end{equation*}
        Up to a subsequence, $E_h \to E$ for some set $E \subset \T^n$ of finite perimeter,
        moreover
        \begin{equation*}
            P_{\T^n}^\phi(E) \le \liminf_{h \to \infty} J_{\gamma_h}^\phi(E_h) \le \liminf_{h \to \infty} J_{\gamma_h}^\phi(L) = P_{\T^n}^\phi(L),
        \end{equation*}
        implying that $E = L$.
        Since $E_h$ minimizes $J_{\gamma_h}^\phi$, we also have
        \begin{equation*}
            0 \ge J_{\gamma_h}^\phi (E_h) - J_{\gamma_h}^\phi (L) 
            = P_{\T^n}^\phi (E_h) - P_{\T^n}^\phi (L) + \gamma_h \big( \F(E_h) - \F(L) \big) > \gamma_h \big( \F(E_h) - \F(L) \big).
        \end{equation*}
        By \cref{cor:2.lamminL1} or \cref{thm:2.lamminL1hor} there exist $\overline \gamma > 0$ and $\delta > 0$ such that
        \begin{equation*}
            J_{\overline \gamma}^\phi (L) < J_{\overline \gamma}^\phi (E) 
        \end{equation*}
        for any subset $E \subset \T^n$ with $\abs{E} = M$ and $0 < d_{L^1}(E, L) < \delta$.
        Since $E_h \to L$, we conclude that for $h \in \N$ sufficiently large it must hold
        \begin{align*}
            0 &> J_{\overline \gamma}^\phi (L) - J_{\overline \gamma}^\phi (E_h) 
            = P_{\T^n}^\phi (L) - P_{\T^n}^\phi (E_h) + \overline \gamma \big( \F(L) - \F(E_h) \big) \\
            &> P_{\T^n}^\phi (L) - P_{\T^n}^\phi (E_h) + \gamma_h \big( \F(L) - \F(E_h) \big) = J_{\gamma_h}^\phi (L) - J_{\gamma_h}^\phi (E_h),
        \end{align*}
        obtaining the desired contradiction.
        \qedhere
        
    \end{proof}

    By the previous proposition, we reduce ourselves to studying the periodic Wulff problem \eqref{eq:2.wulffper}. 
    In the planar case we give a simple condition on a uniformly elliptic surface tension 
    which guarantees the global minimality of the horizontal lamella for a suitable range of masses.
    Actually we are able to classify all minimizers for \eqref{eq:2.wulffper}, which are depicted in \cref{fig:7.wulffmin}.
    In the isotropic case the same kind of result can be found in \cite{HHM99}, which is the one then applied in \cite{AFM13}.

    \begin{thm} \label{thm:7.wulffminell}

        Let $n = 2$ and $\phi$ be a uniformly elliptic surface tension.
        Assume also that
        \begin{equation} \label{eq:7.minilam}
            m_\phi \defeq \phi(e_2) + \phi(- e_2) < \phi(\nu) + \phi(- \nu)
        \end{equation}
        for any $\nu \in \S^1 \setminus \cb{\pm e_2}$.
        Then
        \begin{equation} \label{eq:7.minsuff}
            0 < \frac{m_\phi^2}{4 \abs{W_\phi}} < \frac{1}{2},
        \end{equation}
        and, depending on the volume constraint $M \in (0, 1)$, the following sets are the only minimizers of \eqref{eq:2.wulffper}:
        \begin{equation*}
            \begin{dcases}
                \sqrt{\dfrac{M}{\abs{W_\phi}}} W_\phi \qquad & \text{ if } \quad M \le \dfrac{m_\phi^2}{4 \abs{W_\phi}}, \\
                L_M & \text{ if } \quad \dfrac{m_\phi^2}{4 \abs{W_\phi}} \le M \le 1 - \dfrac{m_\phi^2}{4 \abs{W_\phi}}, \\
                \rb{- \sqrt{\dfrac{1 - M}{\abs{W_\phi}}} W_\phi}^c \qquad & \text{ if } \quad M \ge 1 - \dfrac{m_\phi^2}{4 \abs{W_\phi}},
            \end{dcases}
        \end{equation*}
        where $L_M$ is the horizontal lamella of volume $M$. 
        Moreover, for $M \ne \frac{m_\phi^2}{4 \abs{W_\phi}}, 1 - \frac{m_\phi^2}{4 \abs{W_\phi}}$,
        the minimizer is unique (up to translation).
        
    \end{thm}

    \begin{figure}
        \centering
        \begin{tikzpicture} [scale = 1.7]
            \draw[fill = gray!20] (- 1, - 1) -- (- 1, 1) -- (1, 1) -- (1, - 1) -- cycle;
            \draw[dashed, fill = gray!50, scale = 0.4, shift={(0.6, - 0.5)}] (1, 0.8) to[out = 90, in = 0] (- 0.5, 1.5) to[out = 180, in = 90] 
            (- 2.5, 0.5) to[out = - 90, in = 180] (- 0.5, - 0.5) to[out = 0, in = - 90] cycle;
            \draw[thick, ->](- 1, - 0.00001) -- (- 1, 0.00001);
            \draw[thick, ->](1, - 0.00001) -- (1, 0.00001);
            \draw[thick, ->](- 0.00001, - 1) -- (0.00001, - 1);
            \draw[thick, ->](- 0.00001, 1) -- (0.00001, 1);
            \draw(1.3, 0.8) node{$\T^2$};
            \draw(0, 0) node{$\mu_M W_\phi$};
        \end{tikzpicture}
        \hspace{1 cm}
        \begin{tikzpicture} [scale = 1.7]
            \fill[gray!20] (- 1, - 1) -- (- 1, 1) -- (1, 1) -- (1, - 1) -- cycle;
            \fill[gray!50] (- 1, - 0.4) rectangle (1, 0.4);
            \draw[dashed] (- 1, 0.4) -- (1, 0.4);
            \draw[dashed] (- 1, - 0.4) -- (1, - 0.4);
            \draw (- 1, - 1) -- (- 1, 1) -- (1, 1) -- (1, - 1) -- cycle;
            \draw[thick, ->](- 1, - 0.00001) -- (- 1, 0.00001);
            \draw[thick, ->](1, - 0.00001) -- (1, 0.00001);
            \draw[thick, ->](- 0.00001, - 1) -- (0.00001, - 1);
            \draw[thick, ->](- 0.00001, 1) -- (0.00001, 1);
            \draw(1.3, 0.8) node{$\T^2$};
            \draw(0.5, - 0.1) node{$L_M$};
        \end{tikzpicture}
        \hspace{1 cm}
        \begin{tikzpicture} [scale = 1.7]
            \draw[fill = gray!50] (- 1, - 1) -- (- 1, 1) -- (1, 1) -- (1, - 1) -- cycle;
            \draw[dashed, fill = gray!20, scale = - 0.4, shift={(0.6, - 0.5)}] (1, 0.8) to[out = 90, in = 0] (- 0.5, 1.5) to[out = 180, in = 90] 
            (- 2.5, 0.5) to[out = - 90, in = 180] (- 0.5, - 0.5) to[out = 0, in = - 90] cycle;
            \draw[thick, ->](- 1, - 0.00001) -- (- 1, 0.00001);
            \draw[thick, ->](1, - 0.00001) -- (1, 0.00001);
            \draw[thick, ->](- 0.00001, - 1) -- (0.00001, - 1);
            \draw[thick, ->](- 0.00001, 1) -- (0.00001, 1);
            \draw(1.3, 0.8) node{$\T^2$};
            \draw(0.1, 0.7) node{$(- \mu_{1 - M} W_\phi)^c$};
        \end{tikzpicture}
        \caption{The minimizers of problem \eqref{eq:2.wulffper} for increasing values of the the volume $M$: 
        for small volumes, the minimizer is a rescaling of the Wulff shape, for volumes close to $1/2$ the minimizer must be the horizontal lamella
        and finally, for $M$ close to $1$, the minimizer is equal to the complement of a rescaled copy of the Wulff shape reflected through the origin.}
        \label{fig:7.wulffmin}
    \end{figure}
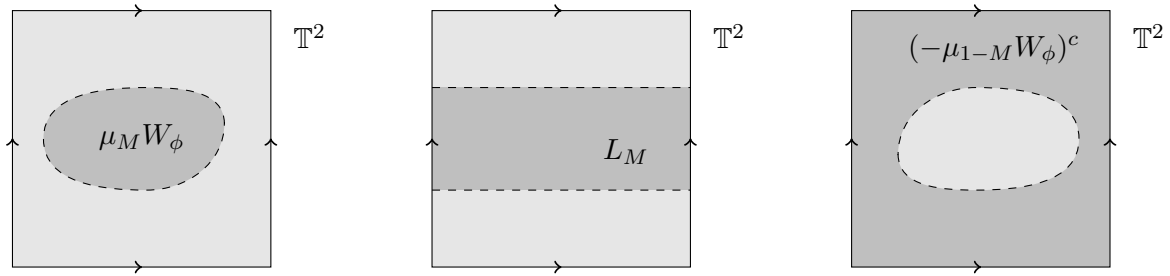

    \begin{proof}
        
        We divide the proof into four steps: we first show that the boundary of a minimizer for \eqref{eq:2.wulffper} is of class $C^2$,
        by some regularity results in the planar case, based on \cite{ANP02}; 
        in the second step we prove the validity of \eqref{eq:7.minsuff} in a stronger form, 
        which we subsequently need to show that a rescaled copy of the Wulff shape is contained in the square $(0, 1)^2$ 
        up to a certain volume; in the last part, we classify the minimizers for \eqref{eq:2.wulffper} by means of the anisotropic mean curvature, 
        finally obtaining the statement of the theorem.
        
        \vspace{5 pt}
        \noindent
        \textit{Step one}: 
        we first observe that, in the planar setting, the boundary of any minimizer $E$ of \eqref{eq:2.wulffper} is of class $C^2$.
        Indeed, $\partial E$ is of class $C^1$ (see for instance \cite[Theorem 6.4; Theorem 6.18]{ANP02}). 
        Then, we fix a generic point $x_0 \in \partial E$ and, up to a rotation and a translation, we may assume that $x_0 = 0$
        and that there exist $r > 0$ and $u \in C^1((- r, r))$ such that
        \begin{equation*}
            B(0, r) \cap \partial E = \cb{(s, u(s)) \in \R^2 : \abs{s} < r}, \qquad \qquad
            B(0, r) \cap E = \cb{(s, t) \in B(0, r) : t < u(s)}.
        \end{equation*}
        Proceeding like in \cref{sec:2}, we have that $P_{\T^2}^\phi(E; B(0, r)) = \int_{- r}^r \phi(- u', 1)$,
        moreover, by volume-constrained minimality of $E$, it holds
        \begin{equation*}
            0 = \dfrac{d}{d t} \rb{\int_{- r}^r \phi \big(- u' + t \varphi', 1 \big)} \restrict{t = 0}
            = \int_{- r}^r \varphi' \partial_1 \phi(- u', 1),
        \end{equation*}
        for any $\varphi \in C_c^\infty((- r, r))$ with $\int_{- r}^r \varphi = 0$.
        In particular, by arbitrariness of $\varphi$, $\big( \partial_1 \phi(- u', 1) \big)' = \lambda$ for some constant $\lambda \in \R$, 
        and \cite[Theorem 7.56]{AFP00} ensures that $u \in W^{2, 2}_\loc((- r, r))$.
        As a consequence,
        \begin{equation*}
            u'' = \frac{\lambda}{\partial_{1 1}^2 \phi(- u', 1)},
        \end{equation*}
        where we note that $\partial_{1 1}^2 \phi(- u', 1) > 0$ by the assumption of uniform ellipticity, 
        proving that $u$ is of class $C^2$ and so is $\partial E$.

        \vspace{5 pt}
        \noindent
        \textit{Step two}: 
        we next prove a stronger form of $\abs{W_\phi} > \frac{m_\phi^2}{2}$, implying the validity of \eqref{eq:7.minsuff}.
        We may find two points $x^+, x^- \in \partial W_\phi$ with $\nu_{W_\phi}(x^\pm) = \pm e_2$,
        moreover, by \cref{lem:2.subsurfchar} one has $\nabla \phi(\pm e_2) = x^\pm$,
        which in turn implies $(x^\pm)_2 = \partial_2 \phi(\pm e_2) = \pm \phi(\pm e_2)$. 
        As a consequence of \eqref{eq:7.minilam} one has that the map $\Phi \in C^2(\R)$ defined by
        \begin{equation*}
            \Phi(t) \defeq \phi \rb{\frac{(t, 1)}{\sqrt{1 + t^2}}} + \phi \rb{\frac{(- t, - 1)}{\sqrt{1 + t^2}}}, \qquad \qquad t \in \R,
        \end{equation*}
        has a minimum at $t = 0$, so that 
        \begin{equation*}
            0 = \Phi'(0) = \partial_1 \phi(0, 1) - \partial_1 \phi(0, - 1).
        \end{equation*}
        In particular, denoting $\eta = \partial_1 \phi(0, 1) = \partial_1 \phi(0, - 1)$, we have that $x^\pm = (\eta, \pm \phi(\pm e_2))$.
        There also exist two points $x^r, x^l \in \partial W_\phi$ such that $(x^r)_1 = \phi(e_1)$ and $(x^l)_1 = - \phi(- e_1)$,
        so that the area of the quadrilateral $Q$ with vertices $x^+, x^-, x^r$ and $x^l$ (see \cref{fig:7.quad}) is
        \begin{equation*}
            \abs{Q} = \dfrac{\big( \phi(e_2) + \phi(- e_2) \big) \big( \phi(e_1) - \eta \big)}{2} 
            + \dfrac{\big( \phi(e_2) + \phi(- e_2) \big) \big( \phi(- e_1) + \eta \big)}{2}
            = \dfrac{m_\phi K_\phi}{2},
        \end{equation*}
        where $K_\phi \defeq \phi(e_1) + \phi(- e_1)$.
        Since $Q \subset \overline W_\phi$ and the latter is strictly convex, we conclude that
        \begin{equation} \label{eq:7.minsuffstr}
            \abs{W_\phi} > \abs{Q} = \frac{m_\phi K_\phi}{2}.
        \end{equation}
        By means of \eqref{eq:7.minilam} we have $K_\phi > m_\phi$, so that $\abs{W_\phi} > \frac{m_\phi^2}{2}$.
        \begin{figure}
            \centering
            \begin{tikzpicture}[scale = 0.8]
                \draw[dashed, fill = gray!50] (1, 0.8) to[out = 90, in = 0] (- 0.5, 1.5) to[out = 180, in = 90] (- 2.5, 0.5) 
                to[out = - 90, in = 180] (- 0.5, - 0.5) to[out = 0, in = - 90] cycle;
                \draw[dashdotted, fill = gray!75] (1, 0.8) -- (- 0.5, 1.5) -- (- 2.5, 0.5)  -- (- 0.5, - 0.5) -- cycle;
                \draw[fill = black] (1, 0.8) circle (0.03) node[right]{$x^r$};
                \draw[fill = black] (- 0.5, 1.5) circle (0.03) node[above]{$x^+$};
                \draw[fill = black] (- 2.5, 0.5) circle (0.03) node[left]{$x^l$};
                \draw[fill = black] (- 0.5, - 0.5) circle (0.03) node[below]{$x^-$};
                \draw[thick, -latex] (0, - 2.5) -- (0, 2.5) node[left]{$x_2$};
                \draw[thick, -latex] (- 3, 0) -- (3, 0) node[above]{$x_1$};
                \draw(0.9, 1.5) node{$W_{\phi}$};
                \draw(- 0.6, 0.6) node{$Q$};
            \end{tikzpicture}
            \hspace{1 cm}
            \begin{tikzpicture}[scale = 0.8]
                \draw[dashdotted, fill = gray!75] (1, 1.5) -- (- 2.5, 1.5) -- (- 2.5, - 0.5) -- (1, - 0.5) -- cycle;
                \draw[dashed, fill = gray!50] (1, 0.8) to[out = 90, in = 0] (- 0.5, 1.5) to[out = 180, in = 90] (- 2.5, 0.5) 
                to[out = - 90, in = 180] (- 0.5, - 0.5) to[out = 0, in = - 90] cycle;
                \draw[thick, -latex](0, - 2.5) -- (0, 2.5) node[left]{$x_2$};
                \draw[thick, -latex](- 3, 0) -- (3, 0) node[above]{$x_1$};
                \draw[|-|] (1.2, - 0.5) -- (1.2, 1.5) node[midway, right] {$m_\phi$};
                \draw[|-|] (- 2.5, - 0.7) -- (1, - 0.7) node[midway, below] {$K_\phi$};
                \draw(- 1.6, 1.8) node{$Q'$};
                \draw(- 0.6, 0.6) node{$W_{\phi}$};
            \end{tikzpicture}
            \caption{The quadrilaterals $Q$ and $Q'$ described in Step two and Step three of the proof of \cref{thm:7.wulffminell}.}
            \label{fig:7.quad}
        \end{figure}
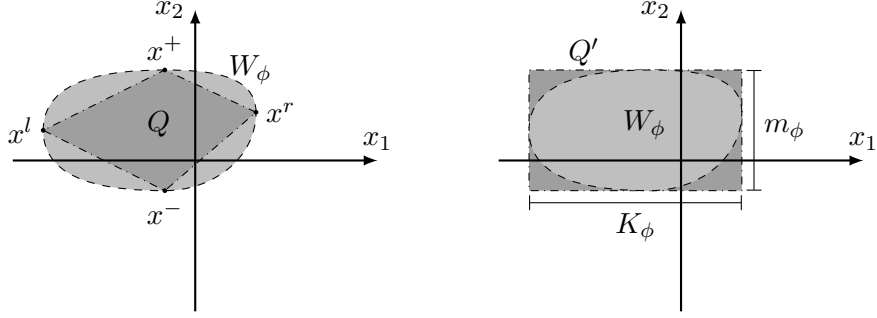
        
        \vspace{5 pt}
        \noindent
        \textit{Step three}: 
        let $\overline M = m_\phi^2/(4 \abs{W_\phi})$ and set $\mu_M \defeq \sqrt{M/\abs{W_\phi}}$ for $M \in (0, 1)$. 
        We show that, for $M \le \overline M$ the set $\mu_M W_\phi$, which is a rescaled copy of the Wulff shape of $\phi$ of volume $M$, 
        is (up to translation) compactly contained in the open square $(0, 1)^2$.
        This amounts to prove that $\mu_{\overline M} \overline W_\phi \subset (0, 1)^2$.
        Indeed, consider the closed rectangle $Q'$ with axis-aligned sides such that $W_\phi$ is inscribed in $Q'$ (see \cref{fig:7.quad}),
        then $Q'$ has basis $K_\phi$ and height $m_\phi$.
        By \eqref{eq:7.minsuffstr} there exists $\varepsilon > 0$ such that $2 \abs{W_\phi} > m_\phi (K_\phi + \varepsilon)$, so that
        \begin{equation*}
            \mu_{\overline M} \overline W_\phi = \frac{m_\phi}{2 \abs{W_\phi}} \overline W_\phi 
            \subset \frac{m_\phi}{2 \abs{W_\phi}} Q' \subset \frac{Q'}{K_\phi + \varepsilon},
        \end{equation*}
        but this last set is a rectangle with axis-aligned sides whose basis and height are strictly smaller than $1$, 
        so that it is (up to a translation) compactly contained in $(0, 1)^2$.

        \vspace{5 pt}
        \noindent
        \textit{Step four}: 
        let $E$ be a minimizer of volume $M \in (0, 1)$ for problem \eqref{eq:2.wulffper}, then by Step one $E$ is a set with boundary of class $C^2$.
        One may consider the first variation of the anisotropic perimeter
        $P_{\T^2}^\phi(E)$ with respect to an initial velocity $T \in C_c^1(\T^2; \R^2)$, given by
        \begin{equation*}
            \partial P_{\T^2}^\phi(E)[T] = \int_{\reb E} H_E^\phi (T \cdot \nu_E) \de \Ha^1.
        \end{equation*}
        where $H_E^\phi : \partial E \to \R$ is the \textit{anisotropic mean curvature} of $E$, defined as
        \begin{equation*}
            H_E^\phi(x) \defeq - \dive_\tau \Big( \nabla \phi \big( \nu_E(x) \big) \Big), \qquad \qquad x \in \partial E,
        \end{equation*}
        with $\dive_\tau$ denoting the tangential divergence on $\partial E$.
        By requiring $T$ to be volume preserving (i.e., $\int_{\reb E} T \de \Ha^1 = 0$)
        the anisotropic mean curvature of a minimizer of \eqref{eq:2.wulffper} is easily seen to be equal to a constant $H \in \R$.
        In the planar case, for an elliptic surface tension $\phi$ we have upper and lower bounds on $H_E^\phi$ 
        in terms of the mean curvature $H_E$:
        there exists a continuous tangent vector field $\tau_E : \partial E \to \R^2$ such that, 
        for any $x \in \partial E$, $\tau_E(x)$ and $\nu_E(x)$ form an orthonormal basis of $\R^2$. 
        Hence,
        \begin{align*}
            H_E^\phi(x) &= - \dive_\tau \Big( \nabla \phi \big( \nu_E(x) \big) \Big)
            = - \partial_\tau \Big( \partial_\tau \phi \big( \nu_E(x) \big) \Big) \\
            &= - \partial^2_{\tau \tau} \phi \big( \nu_E(x) \big) \big( \tau(x) \cdot \partial_\tau \nu_E(x) \big)
            = \partial^2_{\tau \tau} \phi \big( \nu_E(x) \big) H_E(x),
        \end{align*}
        which in turn implies that there exists $\lambda > 0$ with
        \begin{equation} \label{eq:7.animeanbound}
            \dfrac{1}{\lambda} H_E(x) \le H_E^\phi(x) \le \lambda H_E(x)
        \end{equation}
        by uniform ellipticity of $\phi$.
        We consider the periodic extension $E_{\R^2}$ of $E$ to $\R^2$, and we study its geometry in relation with the sign of $H_E^\phi$.
        
        \begin{itemize}
            \item 
                If $H_E^\phi = 0$ then $H_E = 0$ and $\partial E_{\R^2}$ is the union of parallel lines.
                Notice that $E$ is connected as a subset of $\T^2$ 
                (otherwise, translating a connected component as to touch another we would contradict minimality).
                Moreover, the length $\ell(r)$ of any line $r$ in the flat torus (that is, the quotient of a line of $\R^2$) is at least $1$, 
                and denoting by $\nu \in \S^1$ a normal vector to $r$ we see by \eqref{eq:7.minilam} that
                \begin{equation*}
                    P_{\T^2}^\phi(E) = \big( \phi(\nu) + \phi(- \nu) \big) \ell(r) 
                    \ge \phi(\nu) + \phi(- \nu) \ge \phi(e_2) + \phi(- e_2) = P_{\T^2}^\phi(L_M),
                \end{equation*}
                with strict inequality if $\nu \in \S^1 \setminus \cb{\pm e_2}$, so that necessarily $E = L_M$.
            
            \item 
                If $H_E^\phi > 0$, fix one of the connected components of $E_{\R^2}$, denoted by $F \subset \R^2$, 
                and notice that $H_E^\phi = H_F^\phi$. 
                By means of \eqref{eq:7.animeanbound} we have that $H_F \ge H_E^\phi/\lambda > 0$,
                so that $F$ is strictly convex and bounded, 
                which in turn implies that every connected component of $E_{\R^2}$ is bounded and strictly convex.
                Now, if $E \subset \T^2$ had more than one connected component, 
                by regularity of minimizers and translational invariance of the anisotropic perimeter, 
                we may translate a connected component as to touch another, contradicting regularity.
                This means that $E$ is connected as a subset of $\T^2$, and going back to $E_{\R^2}$, we note that
                any of its connected components $F \subset \R^2$ is just a translation of one another, moreover $\abs{F} = \abs{E} = M$.
                In that case we may confront the anisotropic perimeter of $F$ on $\R^2$ directly with $\mu_M W_\phi$, obtaining
                \begin{equation*}
                    P_{\T^2}^\phi(E) = P^\phi(F) \ge P^\phi(\mu_M W_\phi),
                \end{equation*}
                with equality if and only if $F = \mu_M W_\phi$ (up to translation).
                If $M \le \overline M$, by Step three we know that $\mu_M W_\phi \subset (0, 1)^2$ (up to translation) so that
                $P_{\T^2}^\phi(E) \ge P^\phi(\mu_M W_\phi) = P_{\T^2}^\phi(\mu_M W_\phi)$, with equality if and only if 
                $E = \mu_M W_\phi$, which in turn implies exactly that $E = \mu_M W_\phi$.
                Observe also that
                \begin{equation*}
                    P^\phi (\mu_{\overline M} W_\phi) = \mu_{\overline M} P^\phi (W_\phi) 
                    = \frac{m_\phi}{2 \abs{W_\phi}} \cdot (2 \abs{W_\phi}) = m_\phi = P_{\T^2}^\phi(L_{\overline M}).
                \end{equation*}
                If instead $M > \overline M$, we notice that 
                \begin{equation*}
                    P_{\T^2}^\phi(E) \ge P^\phi(\mu_M W_\phi) > P^\phi(\mu_{\overline M} W_\phi) 
                    = P_{\T^2}^\phi(L_{\overline M}) = P_{\T^2}^\phi(L_{M}),
                \end{equation*}
                that is, the set $E$ cannot be a minimizer.
                As a consequence $\mu_M W_\phi$ is a minimizer for \eqref{eq:2.wulffper} if and only if $M \le \overline M$,
                and it is the unique minimizer for $M < \overline M$.
                
            \item 
                If $H_E < 0$, consider the uniformly elliptic surface tension $\tilde \phi (x) \defeq \phi(- x)$, $x \in \R^2$;
                in that case, the complement $E^c = \T^2 \setminus E$ is such that 
                $P_{\T^2}^{\tilde \phi} (E^c) = P_{\T^2}^\phi (E)$, $H_{E^c}^{\tilde \phi} = - H_E^\phi > 0$, and it has volume $\abs{E^c} = 1 - M$.
                On noting that
                \begin{equation*}
                    W_{\tilde \phi} = \bigcap_{y \in \S^1} \Big\{ x \in \R^2 : x \cdot y < \phi(- y) \Big\}
                    = \bigcap_{y \in \S^1} \Big\{ - x \in \R^2 : x \cdot y < \phi(y) \Big\} = - W_\phi,
                \end{equation*}
                we conclude by the previous point that for $1 - M \le \overline M$, that is, $M \ge 1 - \overline M$, it holds
                $E^c = \mu_{1 - M} W_{\tilde \phi} = - \mu_{1 - M} W_\phi$ and $E$ is the unique minimizer if $M > 1 - \overline M$, 
                whereas, if $1 - M > \overline M$, then $E$ cannot be a minimizer.
        
        \end{itemize}
        This concludes the proof.
        \qedhere
       
    \end{proof}    

    \begin{rem}

        If one were to weaken the assumption \eqref{eq:7.minilam} by requiring instead that
        \begin{equation} \label{eq:7.minilam2}
            \phi(e_2) + \phi(- e_2) \le \phi(\nu) + \phi(- \nu) \qquad \qquad \text{ for any } \nu \in \S^1,
        \end{equation}
        by following the same proof we would again get that the horizontal lamella of volume $M$ is a minimizer for 
        $\frac{m_\phi^2}{4 \abs{W_\phi}} \le M \le 1 - \frac{m_\phi^2}{4 \abs{W_\phi}}$, 
        but if $\phi(e_2) + \phi(- e_2) = \phi(e_1) + \phi(- e_1)$ then also the vertical lamella 
        (defined as $(\ell_1, \ell_2) \times \T^1$ for some $\ell_1, \ell_2 \in (0,1)$, $\ell_1 < \ell_2$) 
        is a minimizer for the same range of masses.
        As a consequence, we retrieve the result valid for the isotropic case $\phi = \abs{\cdot}$ presented in \cite{HHM99}.
        Whether the result still holds for the crystalline case (and more in general, for surface tensions which are not uniformly elliptic)
        is still unknown, and requires further study.
        
    \end{rem}

\printbibliography

\end{document}